\documentclass[final,12pt]{elsarticle}
\usepackage[utf8]{inputenc}

\usepackage[letterpaper,hmargin=1.in,vmargin=1.in]{geometry}
\usepackage{graphicx}
\usepackage{amsmath,amssymb,amsthm}
\usepackage{float} % provides the H float modifier option
\usepackage{siunitx} % format numbers and units easily
\usepackage{upquote,textcomp} % stops straight quotes turning into curly quotes in verbatim https://stackoverflow.com/a/1663936/9911203
\usepackage{booktabs} % \toprule \midrule \bottomrule
\usepackage{tikz,tikz-3dplot}
\usetikzlibrary{arrows.meta}
\usetikzlibrary{positioning,calc,arrows.meta,shapes.geometric,fit}
\usetikzlibrary{cd}
\usepackage{bm}
\usepackage{hyperref}

\usepackage{stmaryrd}
\usepackage{caption}
\usepackage{subcaption}

\usepackage{bbm}

\newtheorem{lemma}{Lemma}
\newtheorem{theorem}{Theorem}
\newtheorem{proposition}{Proposition}
\newtheorem{corollary}{Corollary}

\theoremstyle{definition}
\newtheorem{definition}{Definition}

\theoremstyle{remark}
\newtheorem*{remark}{Remark}

\let\originalleft\left
\let\originalright\right
\renewcommand{\left}{\mathopen{}\mathclose\bgroup\originalleft}
\renewcommand{\right}{\aftergroup\egroup\originalright}

\makeatletter
\renewcommand*\env@matrix[1][\arraystretch]{%
  \edef\arraystretch{#1}%
  \hskip -\arraycolsep
  \let\@ifnextchar\new@ifnextchar
  \array{*\c@MaxMatrixCols c}}
\makeatother

\newcommand{\be}{\begin{equation}}
\newcommand{\ee}{\end{equation}}

\newcommand{\quadscript}[5]{{}_{#4}^{#2} {#1}_{#5}^{#3}}
\newcommand{\fracint}[3]{\quadscript{I}{}{#1}{#2}{#3}}
\newcommand{\rlderiv}[3]{\quadscript{D}{RL}{#1}{#2}{#3}}
\newcommand{\cderiv}[3]{\quadscript{D}{C}{#1}{#2}{#3}}

\newcommand{\fraciint}[2]{\quadscript{\mathcal{I}}{}{#1}{}{#2}}
\newcommand{\rlderiiv}[2]{\quadscript{\mathcal{D}}{RL}{#1}{}{#2}}

\newcommand{\dfraciint}[2]{\quadscript{\mathbb{I}}{}{#1}{}{#2}}
\newcommand{\drlderiiv}[2]{\quadscript{\mathbb{D}}{RL}{#1}{}{#2}}

\newcommand{\cnabla}[1]{\quadscript{\nabla}{C}{#1}{}{}}

\newcommand{\reals}{\mathbb{R}}
\newcommand{\R}{\reals}

\newcommand{\Z}{\mathbb{Z}}

\begin{document}
\begin{frontmatter}



\title{Structure-Preserving Discretization of Fractional Vector Calculus using Discrete Exterior Calculus}
\author[label1]{Alon Jacobson}
\ead{Alon.Jacobson@tufts.edu}
\author[label1]{Xiaozhe Hu\corref{cor1}}
\ead{Xiaozhe.Hu@tufts.edu}
\affiliation[label1]{organization={Department of Mathematics},
			addressline={Tufts University},
			city={Medford},
			postcode={02155},
			state={MA},
			country={USA}}
\cortext[cor1]{Corresponding author}


\begin{abstract}
	Fractional vector calculus is the building block of the fractional partial differential equations that model non-local or long-range phenomena, e.g., anomalous diffusion, fractional electromagnetism, and fractional advection-dispersion. In this work, we reformulate a type of fractional vector calculus that uses Caputo fractional partial derivatives and discretize this reformulation using discrete exterior calculus on a cubical complex in the structure-preserving way, meaning that the continuous-level properties $\operatorname{curl}^\alpha \operatorname{grad}^\alpha = \bm{0}$ and $\operatorname{div}^\alpha  \operatorname{curl}^\alpha = 0$ hold exactly on the discrete level. We discuss important properties of our fractional discrete exterior derivatives and verify their second-order convergence in the root mean square error numerically. Our proposed discretization has the potential to provide accurate and stable numerical solutions to fractional partial differential equations and exactly preserve fundamental physics laws on the discrete level regardless of the mesh size.
\end{abstract}



\begin{keyword}
Fractional vector calculus \sep Discrete exterior calculus \sep Structure-preserving discretization \sep Fractional partial differential equations

\MSC[2020] 65M99 \sep 65N99 \sep 26A33 \sep 35R11
\end{keyword}

\end{frontmatter}

\section{Introduction}
Fractional calculus generalizes the integer order integration and differentiation to non-integer order. Unlike standard derivatives and integrals, fractional derivatives and integrals are non-local operators, enabling them to model long-range dependence. In this work, we focus on fractional vector calculus (FVC), which analogously extends the vector calculus to fractional order. Fractional calculus and FVC are widely used in fractional partial differential equations (FPDEs), which recently have a wide range of new scientific and engineering applications. For example, fractional diffusion equations model anomalous diffusion \cite{metzler1994fractional,SokolovAnomalousDiffusion,dos2019analytic,oliveira2019anomalous,evangelista2018fractional}, fractional Maxwell's equations generalize Maxwell's equations to fractional order \cite{TARASOV20082756,baleanu2009fractional,ortigueira2015fractionalMaxwell}, fractional advection-dispersion equations describe subsurface transport \cite{benson2000application,pang2019fpinns,d2014multidimensional,meerschaert2006fractional}, fractional Laplacians are used in image processing \cite{gatto2015numerical}, fractional differential equations are used in finance \cite{scalas2000fractional}, and a fractional gradient has been used for fractional backpropagation in training neural networks \cite{wei2020generalization}.

There are various definitions of FVC, each with their own strengths and weaknesses. Many approaches use a fractional partial derivative in each coordinate direction to construct a fractional nabla operator. Other approaches use an anisotropic mixture of fractional directional derivatives in each direction via an integral, while still other approaches use an isotropic mixture of function values throughout Euclidean space to define the operators.

Due to the complexity of FPDEs, the solutions cannot usually be computed symbolically, so numerical approximations are essential for solving them. Various finite-element and finite-difference methods have been developed for the discretization of FVC to be used in solving these FPDEs, with techniques including finite-difference methods \cite{wang2010direct}, the discretization of fractional directional derivatives \cite{meerschaert2004vector,PANG2013597,pang2019fpinns}, spectral decompositions \cite{song2017computing} and physics-informed neural networks \cite{pang2019fpinns}.

When solving PDEs and FPDEs numerically, some consideration must be given to computational efficiency (i.e., how much time or computer memory is required), as well as the accuracy of the solution obtained (i.e., how close the numerical approximation is to the true solution). Another property that is often desirable is to have chosen continuous-level properties of the model be satisfied \emph{exactly} in its discretization. Such discretizations are termed \emph{structure-preserving}. One possible structure to preserve is to preserve the de Rham exact sequence, which essentially means preserving the vector calculus identities $\operatorname{curl} \operatorname{grad} f = \bm{0}$ and $\operatorname{div} \operatorname{curl} \bm{F} = 0$ exactly in the discretization. The de Rham exact sequence plays an important role in many physical laws, such as incompressibility and Gauss's law of magnetism.

One way to preserve this de Rham exact sequence is by using discrete exterior calculus (DEC). DEC is a computational toolkit that creates discrete operators and definitions that are analogous to the corresponding operators from multivariate calculus. It has recently been gaining popularity as a tool for developing numerical methods for solving PDEs in computational simulations, such as mechanics problems \cite{leok2004foundations}, Lie advection \cite{mullen2011discrete}, and computational fluid dynamics \cite{elcott2007stable}. In addition to being used as a structure-preserving finite-element method, DEC is also widely used in other areas such as computer graphics applications \cite{dominitz2009texture} and geometry processing applications \cite{hormann2007mesh}.

In DEC, the \emph{discrete exterior derivative operator}, $\mathbb{D}_p$, is the discrete version of grad, curl, and div for $p=0$, $p=1$, and $p=2$, respectively. $\mathbb{D}_p$ is a $n_{p+1} \times n_p$ matrix, where $n_p$ is the number of $p$-cells in the complex (for more detail, see Section~\ref{prelim:DEC}). DEC preserves the de Rham exact sequence because the discrete exterior derivative operators satisfy $\mathbb{D}_{p+1} \mathbb{D}_p = 0$ for $p \geq 0$, which is the discrete version of $\operatorname{curl} \operatorname{grad} f = \bm{0}$ and $\operatorname{div} \operatorname{curl} \bm{F} = 0$ for $p=0$ and $p=1$, respectively.

Many types of fractional vector calculus possess an analogous exact sequence $\operatorname{curl}^\alpha \operatorname{grad}^\alpha f = \bm{0}$ and $\operatorname{div}^\alpha \operatorname{curl}^\alpha \bm{F} = 0$. However, to the best of our knowledge, no discretization of FPDEs or FVC preserves this exact sequence. Additionally, despite the usefulness of DEC and the applicability of fractional calculus and fractional vector calculus, there is rarely any work on formulating a \emph{fractional discrete exterior calculus} (FDEC), which generalizes DEC to a fractional order.

To the best of our knowledge, the only existing work on FDEC is \cite{CRUM201964}, which considered the following ``two-sided'' fractional Caputo derivative of order $\alpha \in (0,1)$ of a function $f \in C^1[a,b]$ in 1D:
\begin{equation} \label{def:2-side-Caputo-Derivative}
	D^\alpha f(x) := \frac{1}{\Gamma(1-\alpha)} \int_{[a,b] \setminus \{x\}}
	\frac{f'(\tau)}{|x-\tau|^{\alpha}} \, d\tau.
\end{equation}
\cite{CRUM201964} then defined a \emph{fractional discrete exterior derivative} by discretizing \eqref{def:2-side-Caputo-Derivative}, and then generalizing the resultant discrete operator to higher dimensions. This results in the following $n_{p+1} \times n_p$ matrix, 
\[
\mathbb{D}_p^\alpha = W_{p+1}^{1-\alpha} \mathbb{D}_p,
\]
where $W_{p+1}^{1-\alpha} \in \mathbb{R}^{n_{p+1} \times n_{p+1}}$ is a straightforward generalization of the discretization of the fractional integration of order $1-\alpha$ present in \eqref{def:2-side-Caputo-Derivative} to higher dimensions. ($\mathbb{D}_0^\alpha$ is a discretization of \eqref{def:2-side-Caputo-Derivative} when the complex is one-dimensional.) Unfortunately, the FDEC introduced in \cite{CRUM201964} does not satisfy the fractional generalization of the property $\mathbb{D}_{p+1} \mathbb{D}_p = 0$, i.e., $\mathbb{D}_{p+1}^\alpha \mathbb{D}_p^\alpha \neq 0$. Therefore, even if it \emph{is} a discretization of some type of FVC, it cannot possibly be a discretization that preserves a fractional de Rham exact sequence. Furthermore, in \cite{CRUM201964}, since the FDEC in the higher dimensional case is obtained directly from the 1D case, it is not clear whether it is indeed a discretization of any FVC anymore in higher dimensions, which may limit its potential in numerical simulations involving FPDEs in higher dimensions. Indeed, while numerical experiments verified the expected result that $\mathbb{D}_0^\alpha$ converges to $D^{\alpha}$ in one dimension, convergence in higher dimensions was not shown; although the authors compared $\mathbb{D}_0^\alpha$ to a 2-sided Caputo gradient field of a scalar-valued function in 2-d, convergence with decreasing mesh size was not shown in this case. This is understandable, since due to their definition of the fractional discrete exterior derivative, one would not expect that $\mathbb{D}_0^\alpha$ should converge to this fractional gradient field.

\subsection{Contributions}
Our goal in this work is to define an FDEC that does not suffer the abovementioned problems. Namely, we want our FDEC operators to (1) be direct discretizations of a type of FVC that possesses the exact sequence $\operatorname{curl}^\alpha \operatorname{grad}^\alpha = \bm{0}$ and $\operatorname{div}^\alpha \operatorname{curl}^\alpha = 0$, and (2) be structure-preserving, by having the corresponding exact sequence $\mathbb{D}_{p+1}^\alpha \, \mathbb{D}_p^\alpha = 0$. Both of these properties are achieved by rewriting the operators from a type of FVC that does have the exact sequence (namely, one defined by a fractional nabla operator) 
as compositions of fractional integration and exterior derivatives on the continuous level, and then discretizing these composite operators using DEC on a regular cubical complex. This results in the following fractional discrete exterior derivative operators,
\[ \mathbb{D}_p^\alpha = \dfraciint{1-\alpha}{p+1} \, \mathbb{D}_p \, (\dfraciint{1-\alpha}{p})^{-1}, \qquad p=0,1,2, \quad 0<\alpha<1, \]
where the matrix $\dfraciint{1-\alpha}{p}$ is a discretization of $p$-dimensional fractional integration of order $1-\alpha$ on the $p$-cells.

By discretizing a type of FVC, these operators can be implemented in software and open doors for numerically solving FPDEs. Furthermore, since our approach is structure-preserving -- satisfying the continuous-level properties $\operatorname{curl}^\alpha \operatorname{grad}^\alpha = \bm{0}$ and $\operatorname{div}^\alpha \operatorname{curl}^\alpha = 0$ exactly on the discrete level -- our proposed FDEC operators have high accuracy in discretizing the corresponding continuous operators and behave similarly even at coarse mesh sizes, and can potentially increase the fidelity of numerical solutions to FPDEs. In addition, unlike the usual dense matrices obtained from discretizing the fractional derivatives, the matrices involved in our discretization are relatively sparse, which enables fast computations. Finally, since these operators are extensions of DEC, they can provide fractional generalizations of applications that use DEC.
%%%%%%%%%%%%%%%%%%%%%%%%%%%%%%%%%%%%

\subsection{Outline}
An outline of the paper is as follows. In Section~\ref{sec:preliminaries}, we recall fractional calculus, fractional vector calculus, and discrete exterior calculus. In Section~\ref{sec:FDEC}, we present and prove our reformulation of Tarasov's FVC and describe its discretization. In Section~\ref{sec:numerics}, we show the convergence of our FDEC to Tarasov's FVC numerically, and finally we summarize our work and suggest possible future work in Section~\ref{sec:conclusion}.

\section{Preliminaries}\label{sec:preliminaries}
In this section, we briefly recall fractional calculus first, and then introduce the fractional vector calculus and discrete exterior calculus, which are the building blocks of our FDEC.

\subsection{Fractional calculus} \label{prelim:FC}
First, we discuss fractional calculus. There are many different definitions. This paper focuses on the Riemann-Liouville fractional integral and the Caputo and Riemann-Liouville fractional derivatives. Rather than defining the usual fractional integrals and derivatives, below we define ``partial'' fractional integrals and derivatives of a scalar-valued function of multiple variables, $f : \Omega \to \R$, where $\Omega = [x^1_{\min}, x^1_{\max}] \times \cdots \times [x^m_{\min}, x^m_{\max}] \subset \mathbb{R}^m$. This generalizes the usual definitions, in the sense that if $m=1$, then these definitions reduce to the usual one-dimensional definitions. Furthermore, definitions of partial fractional integrals and derivatives that are essentially the same as the ones we will define can be found in \cite{kilbas2006theory}.
\subsubsection{Fractional integrals and derivatives}
For a real number $\alpha > 0$ and a real-valued function $g : [a,b] \to \R$, the left-sided Riemann-Liouville fractional integral of order $\alpha$ is defined as follows:
\[
\fracint{\alpha}{a}{x}[x'] g(x') = \frac{1}{\Gamma(\alpha)} \int_a^x \frac{dx'}{(x-x')^{1-\alpha}} g(x') \quad (a \leq x \leq b)
\]
where $\Gamma(\cdot)$ denotes the gamma function. Similarly, the left-sided Riemann-Liouville partial fractional integral with respect to coordinate $x^j$ from $a$ to $b$ of order $\alpha$ of a function $f : \Omega \to \R$ is defined as follows:
\begin{align*}
&\fracint{\alpha}{a}{b,x^j}[x'] f := \fracint{\alpha}{a}{b}[x'] f(x^1,\dots,x^{j-1},x',x^{j+1},\dots,x^m) \\
& \quad = \frac{1}{\Gamma(\alpha)} \int_a^{b} \frac{f(x^1,\dots,x^{j-1},x',x^{j+1},\dots,x^m)}{(b-x')^{1-\alpha}} \, dx', \quad (x_{\min}^j \leq a \leq b \leq x_{\max}^j).
\end{align*}
Also, we define
\[
\fracint{\alpha}{}{x^j} f(x^1,\dots,x^m) := \fracint{\alpha}{x_{\min}^j}{x^j,x^j}[x'] f.
\]
Next, we recall the left-sided Riemann-Liouville derivative and the left-sided Caputo derivative. If we let $D^n_{x^j}$ denote the $n$th partial derivative with respect to coordinate $x^j$ (we drop the superscript when $n=1$), then the left-sided Riemann-Liouville fractional partial derivative of $f$ with respect to coordinate $x^j$ at a point $(x^1,\dots,x^m)$ of order $\alpha \geq 0$ is defined as, for $n = \lfloor \alpha \rfloor + 1$,
\begin{align*}
	\rlderiv{\alpha}{}{x^j}f := D^n_{x^j} \fracint{n-\alpha}{}{x^j} f = \frac{1}{\Gamma(n-\alpha)} \left(\frac{\partial}{\partial x^j}\right)^n \int_{x^j_{\min}}^{x^j} \frac{f(x^1,\dots,x^{j-1},x',x^{j+1},\dots,x^m)}{(x^j-x')^{1-(n-\alpha)}} \, dx'.
\end{align*}
Naturally,  $\rlderiv{k}{}{x^j}f = D^k_{x^j} f$ if $k \in \mathbb{N}_0 := \{0,1,2,\dots\}$.

Similarly, for $0 < \alpha \notin \mathbb{N}_0$, the left-sided Caputo fractional partial derivative with respect to coordinate $x^j$ of order $\alpha$ is defined as, for $n = \lfloor \alpha \rfloor + 1$,
\begin{align*}
\cderiv{\alpha}{}{x^j}f := \fracint{n-\alpha}{}{x^j} D^n_{x^j}f = \frac{1}{\Gamma(n-\alpha)} \int_{x_{\min}^j}^{x^j} \frac{D^n_{x^j}f(x^1,\dots,x^{j-1},x',x^{j+1},\dots,x^m)}{(x^j-x')^{1-(n-\alpha)}} \, dx'.
\end{align*}
Otherwise, for integer orders, we define $\cderiv{k}{}{x^j}f := D^k_{x^j} f$ for $\alpha = k \in \mathbb{N}_0$.

\subsubsection{Fractional calculus identities}
Here we will present some identities involving the fractional derivatives and integrals defined above that will be used in this work. The first such identity is named the ``fundamental theorem of fractional calculus'' (FTFC) by \cite{TARASOV20082756}, which generalizes the fundamental theorem of calculus to fractional order. Only the first part of the FTFC will be presented, since the other part is not necessary for the results of this paper.

The first part of the FTFC states that both the Caputo and Riemann-Liouville derivatives are left inverse operators of the Riemann-Liouville integration operator from the left. This generalizes the well-known formula $\frac{d}{dx}\int_a^x f(t) \, dt = f(x)$.
\begin{lemma}\label{lem:FTFC}
Let $f : \Omega \subset \R^m \to \R$ be continuous and let $\alpha > 0$. Then at any point $(x^1,\dots,x^m) \in \Omega$ and any $j=1,\dots,m$,
\[
\rlderiv{\alpha}{}{x^j} \fracint{\alpha}{}{x^j} f = f \quad \text{and} \quad
\cderiv{\alpha}{}{x^j} \fracint{\alpha}{}{x^j} f = f.
\]
\end{lemma}
\begin{proof}
Since $f$ is continuous, it is continuous in each variable separately. Then the first equality follows from Lemma 2.4 on page 74 and Lemma 2.9 (b) on page 77 of \cite{kilbas2006theory}, and the second equality follows from Lemma 2.21, part (a) on page 95 of \cite{kilbas2006theory}.
\end{proof}
We also have the following result which will be used later:
\begin{lemma}\label{lem:IRLD}
Let $f : \Omega \subset \R^m \to \R$ be continuous and let $\alpha \in (0,1)$. Then at any point $(x^1,\dots,x^m) \in \Omega$ and any $j=1,\dots,m$,
\[
\fracint{\alpha}{}{x^j} \rlderiv{\alpha}{}{x^j} f = f.
\]
\end{lemma}
\begin{proof}
The proof uses the FTFC (Lemma \ref{lem:FTFC}) presented above:
\begin{equation*}
\fracint{\alpha}{}{x^j} \rlderiv{\alpha}{}{x^j} f = \fracint{\alpha}{}{x^j} D_{x^j} \fracint{1-\alpha}{}{x^j} f = \cderiv{1-\alpha}{}{x^j} \fracint{1-\alpha}{}{x^j} f = f.
\end{equation*}
\end{proof}
The third identity that we will present is that the Riemann-Liouville fractional integral satisfies the so-called \emph{semigroup property}:
\begin{lemma}[Theorem 2.2, \cite{fracdiffeq}]\label{lem:semigroup}
Let $f : \Omega \subset \R^m \to \R$ be continuous and let $\alpha > 0$, $\beta > 0$. Then for any $j=1,\dots,m$ and any $a,b \in [x^j_{\min}, x^j_{\max}]$ with $a \leq b$,
\[
\fracint{\alpha}{a}{b}[x'] \fracint{\beta}{a}{x',x^j}[x''] f = \fracint{\alpha+\beta}{a}{b,x^j}[x'] f.
\]
\end{lemma}

\subsection{Fractional vector calculus}
There are various approaches to defining FVC. Many approaches are similar to the integral order: use the standard basis of the vector space $V$ and represent the fractional gradient of a scalar function $f$ on $V$ by an $n$-tuple of the one-dimensional partial fractional derivatives along the coordinate axes. These operators, which we will term \emph{standard basis directional}, are the operators we discretize in this work (specifically, Tarasov's). Other approaches use an anisotropic mixture of fractional directional derivatives in each direction via an integral, and these operators are referred to as \emph{directional}. Yet a third approach is to use an isotropic mixture of the function values throughout $\R^n$ to define the operators, and these operators are referred to as \emph{isotropic} (termed ``Cartesian'' in \cite{UnifiedFVC}). Finally, \emph{nonlocal vector calculus} generalizes isotropic vector calculus by using an integral with an arbitrary interaction kernel to define a nonlocal gradient and divergence. Both standard basis directional operators \cite{meerschaert2006fractional,d2014multidimensional} and isotropic operators \cite{UnifiedFVC} are special cases of directional operators.

\subsubsection{Tarasov's fractional vector calculus}
In \cite{TARASOV20082756}, FVC operators using the one-dimensional (left-sided) Caputo derivative are defined. Considering a parallelepiped $\Omega = [x^1_{\min},x^1_{\max}] \times [x^2_{\min},x^2_{\max}] \times [x^3_{\min},x^3_{\max}]$, Tarasov's FVC (T-FVC) is based on the following generalization of the nabla operator to fractional order $\alpha > 0$:
\begin{equation}\label{def:frac-nabla-op}
	\cnabla{\alpha} := \mathbf{e}_1 \, \cderiv{\alpha}{}{x^1} + \mathbf{e}_2 \, \cderiv{\alpha}{}{x^2} + \mathbf{e}_3 \, \cderiv{\alpha}{}{x^3}.
\end{equation}
For a scalar-valued function $f$ and vector-valued function $\bm{F}$, the T-FVC operators are then defined as,
\begin{equation} \label{def:T-FVC-operators}
\operatorname{grad}_T^\alpha f := \cnabla{\alpha} f, \quad 
\operatorname{curl}_T^\alpha \bm{F} := \cnabla{\alpha} \times \bm{F}, \quad \operatorname{div}_T^\alpha \bm{F} := \cnabla{\alpha} \cdot \bm{F}.
\end{equation}
A property of the T-FVC operators is that they satisfy fractional generalizations of Green's, Stokes', and Gauss's theorem, which use fractional line and surface integrals, see \cite{TARASOV20082756} for details.

%%%%%%%%%%%%%%%%%%%%% ADD BACK TARASOV LAPLACIAN? %%%%%%%%%%%%%%%%%%%%%
%By composing the fractional divergence with the fractional gradient, this leads to a fractional Laplacian operator:
%\[
%\operatorname{Div}_R^\alpha \operatorname{Grad}_R^\alpha f(x,y,z) =
%(\cderiv{\alpha}{}{R}[x])^2 f(x,y,z) +
%(\cderiv{\alpha}{}{R}[y])^2 f(x,y,z) +
%(\cderiv{\alpha}{}{R}[z])^2 f(x,y,z)
%,
%\]
%which is an anisotropic operator of order $2\alpha$.

Another important property of the operators defined above, which is necessary for our goal to make structure-preserving FDEC operators, is that they satisfy 
\begin{equation}\label{eqn:frac-exactness}
	\operatorname{curl}_T^\alpha \, \operatorname{grad}_T^\alpha = \bm{0} \quad \text{and} \quad  \operatorname{div}_T^\alpha \, \operatorname{curl}_T^\alpha = 0,
\end{equation}
which are generalizations of the identities from vector calculus $\operatorname{curl} \operatorname{grad} = \bm{0}$ and $\operatorname{div} \operatorname{curl} = 0$. Property~\eqref{eqn:frac-exactness} is the main reason we choose to discretize the T-FVC operators. We plan to retain these properties on the discrete level to produce structure-preserving FDEC operators. This is feasible via rewriting these operators, and the details will be discussed later in Section~\ref{sec:FDEC}.

\subsubsection{Other fractional vector calculus}
From the fractional nabla operator \eqref{def:frac-nabla-op}, we can see that the T-FVC operators are standard basis directional. In \cite{ortigueira2018fractional}, another standard basis directional fractional vector calculus was introduced. Different from Tarasov's work, both left-sided and right-sided FVC operators were introduced. In addition, each coordinate can have different fractional orders, which makes their FVC framework more flexible. This is done by introducing left-sided and right-sided fractional nabla operators first and then defining two sets of fractional gradient, curl, and divergence operators accordingly in the same manner as the T-FVC operators.

An anisotropic, directional FVC was first present in \cite{meerschaert2006fractional}. The authors defined a fractional integration using a mixing measure, a positive finite measure on the set of unit vectors. This considers the relative strength of the dispersion in each radial direction and, hence, introduces anisotropy. The fractional gradient is defined by taking the usual gradient and then the fractional integral. On the other hand, the fractional curl and divergence operators are defined in a reverse order, i.e., taking the fractional integral first and then the usual curl and divergence. \cite{meerschaert2006fractional} and \cite{d2014multidimensional} point out that if the mixing measure is a point mass at each coordinate, the T-FVC gradient is a special case of the directional gradient defined in \cite{meerschaert2006fractional}. In addition, \cite{d2014multidimensional} further generalized the FVC in \cite{meerschaert2006fractional} by allowing the fractional order to vary by direction.

As for isotropic FVC, there is generally a single kind of isotropic fractional vector calculus that appears repeatedly in the literature, for example, Refs. \cite{fracPDEs1,fracPDEs2,vsilhavyfractional}, and can take on various functional forms which are all equivalent. Generally, authors only define a fractional gradient, but in \cite{vsilhavyfractional}, a fractional divergence is also defined. \cite{vsilhavyfractional} points out that due to the construction, FVC constructed using the standard basis directional approach depends on the chosen coordinate system, which could be a drawback since the resulting FVC operators do not transform under rotations. \cite{vsilhavyfractional} introduces an isotropic FVC that satisfies specific transformation rules for translation, rotation, and scaling, and they also generalize the definitions to any real-number fractional order. Unfortunately, although a fractional gradient and divergence are defined, a fractional curl is currently missing from isotropic FVC.

\emph{Nonlocal vector calculus} is any type of generalization of vector calculus that uses an integral with an interaction kernel to make vector calculus nonlocal. \cite{UnifiedFVC} presents a unified framework for nonlocal vector calculus, of which isotropic FVC is a special case. In addition, the authors prove that the directional and isotropic fractional gradients and divergences are equivalent up to constants if the mixing measure is constant. Just as with isotropic FVC, a nonlocal curl is not defined in nonlocal FVC.

Finally, although we only consider the T-FVC in this paper (which is a special case of directional FVC), we believe our approach can be generalized to more general cases, which is a subject of our ongoing work and will be reported in the future.

\subsection{Discrete exterior calculus} \label{prelim:DEC}
Since our discretization of FVC employs DEC, we give a short description of DEC following the methodology developed in~\cite{hirani2003discrete,grady2010discrete,desbrun2005discrete,gillette2009notes}.

Let us start with exterior calculus, which is essentially a generalization of vector calculus to more than three dimensions. Scalar and vector fields from vector calculus are replaced with $p$-forms. From a geometric perspective, a differential $p$-form can be viewed as an oriented $p$-dimensional density \cite{teixeira2013differential}. We denote a $p$-form as $\omega^p$ and the space of $p$-forms as $\Lambda^p$. If we define ${\mathfrak{J}}_{p,m}:=\{J=(i_{1},\ldots ,i_{p}):1\leq i_{1}<i_{2}<\cdots <i_{p}\leq m\}$, then the $p$-forms $\{dx^J\}_{J \in \mathfrak{J}_{p,m}}$ span the space of differential $p$-forms, where we denote  $dx^{J}:=dx^{i_{1}}\wedge \cdots \wedge dx^{i_{p}}=\bigwedge _{i\in J}dx^{i}$. Thus, any $p$-form can be written as
\[
\omega^p = \sum_{J \in \mathfrak{J}_{p,m}} a_J dx^J,
\]
where each $a_J$ is a scalar-valued function of the coordinates. The exterior derivative, an important operation on differential forms, is denoted as $d_p : \Lambda^p \to \Lambda^{p+1}$.
It is defined by:
    \[
    d_p \left( \sum_{J \in \mathfrak{J}_{p,m}} a_J dx^J \right) = \sum_{J \in \mathfrak{J}_{p,m}} \sum_{i=1}^m \frac{\partial a_J}{\partial x^i} dx^i \wedge dx^J.
    \]
In 3D, we have $d_0 = \operatorname{grad}$, $d_1 = \operatorname{curl}$, and $d_2 = \operatorname{div}$. The capstone of exterior calculus is Stokes's theorem, which is a generalization of the fundamental theorem of calculus. It states that under relatively mild smoothness requirements on a compact oriented $(p+1)$-dimensional manifold $\mathcal{S}$ with boundary $\partial \mathcal{S}$,
\begin{equation}\label{def:stokes-thm}
\int_{\mathcal{S}} d_p \, \omega^p = \int_{\partial \mathcal{S}} \omega^p, \quad \omega^p \in \Lambda^p.
\end{equation}

DEC is a discrete analog of exterior calculus. A \emph{$p$-cell}, i.e., a cell of order $p$, denoted by $\sigma^p$, may be represented as an ordered set of vertices comprising a convex $p$-polytope. For example, the $0$- $1$-, $2$-, and $3$-cells are called nodes, edges, faces, and volumes, respectively. A \emph{cell complex of dimension $n$}, or an \emph{$n$-complex}, is a collection of cells of order at most $n$ that obeys certain properties regarding how the cells are connected to each other. We denote the number of $p$-cells in the complex as $n_p$. Each $p$-cell is \emph{oriented} and may have one of two possible orientations. A node ($0$-cell) has two orientations, ``sourceness'' or ``sinkness''. The orientation of an edge ($1$-cell) corresponds to a notion of direction, the orientation of a face ($2$-cell) corresponds to a notion of clockwise/counterclockwise, and the orientation of a volume ($3$-cell) corresponds to a notion of outward/inward.

A $p$-chain $\boldsymbol{\tau}_p$ represents a domain of integration, and is a formal sum of the $p$-cells with coefficients in $\mathbb{Z}$, i.e., $\boldsymbol{\tau}_p = \sum_{i=1}^{n_p} a_i \sigma^p_i$, $\sigma^p_i \in C_p$ and $a_i \in \Z$, where we denote the vector space of $p$-chains as $C_p$. Note that the set of $p$-cells forms a basis for $C_p$. Without confusion, we abuse the notation and also use a vector representation, i.e., $\boldsymbol{\tau}_p = [a_1, \dots a_{n_p}]^\top$. In addition, a $p$-cochain, also known as a discrete $p$-form, is a linear map $\mathbf{c}^p$ from $C_p$ to $\R$. The vector space of $p$-cochains is denoted $C^p$. The \emph{natural pairing}  of a $p$-cochain $\mathbf{c}^p$ and a $p$-chain $\boldsymbol{\tau}_p$ is defined as
\[
\llbracket \mathbf{c}^p, \boldsymbol{\tau}_p \rrbracket :=
\mathbf{c}^p(\boldsymbol{\tau}_p) =
\mathbf{c}^p \left ( \sum_{i=1}^{n_p} a_i \sigma^p_i \right )
=
\sum_{i=1}^{n_p} a_i \mathbf{c}^p (\sigma^p_i). 
\]
Therefore, we can identify a $p$-cochain as a vector $\mathbf{c}^p = [\mathbf{c}^p (\sigma^p_1), \dots, \mathbf{c}^p (\sigma^p_{n_p})]^\top$, which implies that $\llbracket \mathbf{c}^p, \boldsymbol{\tau}_p \rrbracket = (\mathbf{c}^p) ^\top \boldsymbol{\tau}_p$.

Now we introduce the \emph{coboundary operators} or \emph{discrete exterior derivative operators}, $\mathbb{D}_p : C^p \to C^{p+1}$ for $0 \leq p \leq n-1$, which are the discrete versions of the exterior derivatives $d_p : \Lambda^p \to \Lambda^{p+1}$, and can be represented as incidence matrices $\mathbb{D}_p \in \mathbb{R}^{n_{p+1} \times n_p}$ which are defined as
\[
\mathbb{D}_p(i,j) = \begin{cases}
	0 & \text{if $\sigma^p_j$ is not on the boundary of $\sigma^{p+1}_i$,} \\
	+1 & \text{if $\sigma^p_j$ is coherent with the induced orientation of $\sigma^{p+1}_i$,} \\
	-1 & \text{if $\sigma^p_j$ is not coherent with the induced orientation of $\sigma^{p+1}_i$}.
\end{cases}
\]
Furthermore, we denote the \emph{boundary operator} as $\partial_{p} : C_p \to C_{p-1}$, which satisfies $\mathbb{D}_p = \partial_{p+1}^\top$.

An important property of the coboundary and boundary operators is that $\mathbb{D}_{p+1} \mathbb{D}_p = 0$ and $\partial_{p+1} \partial_{p+2} = 0$, for $0 \leq p \leq n-2$. In addition, the DEC version of Stokes' theorem can be stated as follows:
\begin{equation}\label{def:DEC-stokes-thm}
\llbracket \mathbb{D}_p \mathbf{c}^p, \boldsymbol{\tau}_{p+1} \rrbracket = \llbracket \mathbf{c}^p, \partial_{p+1} \boldsymbol{\tau}_{p+1} \rrbracket.
\end{equation}
Let $\mathcal{R}_p : \Lambda^p \to C^p$ be the $p$-th \emph{de Rham map}, which is defined as, for $\omega^p \in \Lambda^p$ and $\boldsymbol{\tau}_p \in C_p$, $
\llbracket \mathcal{R}_p \omega^p , \boldsymbol{\tau}_p \rrbracket = (\mathcal{R}_p \omega^p)(\boldsymbol{\tau}_p) := \int_{\boldsymbol{\tau}_p} \omega^p$.  Using \eqref{def:stokes-thm} and \eqref{def:DEC-stokes-thm}, we have that for any $\omega^p \in \Lambda^p$ and $\boldsymbol{\tau}_{p+1} \in C_{p+1}$,
\begin{align*}
&\llbracket \mathbb{D}_p \mathcal{R}_p \omega^p, \boldsymbol{\tau}_{p+1} \rrbracket = \llbracket \mathcal{R}_p \omega^p, \partial_{p+1} \boldsymbol{\tau}_{p+1} \rrbracket  =
\int_{\partial_{p+1} \boldsymbol{\tau}_{p+1}} \omega^p =
\int_{\boldsymbol{\tau}_{p+1}} d_p \omega^p =
\llbracket \mathcal{R}_{p+1}d_p\omega^p, \boldsymbol{\tau}_{p+1} \rrbracket,
\end{align*}
which implies
\begin{equation*} %\label{dpDiscretization}
\mathbb{D}_p \mathcal{R}_p = \mathcal{R}_{p+1} d_p,
\end{equation*}
and the following commutative diagram,
\begin{center}
	\begin{tikzcd}
		\Lambda^p \arrow[r, "d_p"] \arrow[d, "\mathcal{R}_p"] &  \Lambda^{p+1} \arrow[d, "\mathcal{R}_{p+1}"] \\
		C^p \arrow[r, "\mathbb{D}_p"]           &  C^{p+1}
	\end{tikzcd}
\end{center}
which essentially says that $\mathbb{D}_p$ is the discretization of $d_p$. For example, in 3D, $\mathbb{D}_0$, $\mathbb{D}_1$, and $\mathbb{D}_2$ are discretizations of the gradient, the curl, and the divergence, respectively. In addition, such a discretization is \emph{structure-preserving} since we have $d_{p+1} d_p = 0$ on the continuous level and $\mathbb{D}_{p+1} \mathbb{D}_p = 0$ on the discrete level. In Section~\ref{sec:FDEC}, we define our FDEC operators to be structure-preserving as well, satisfying the analogous property that $\mathbb{D}_{p+1}^\alpha \mathbb{D}_p^\alpha = 0$.

Finally, we introduce the following generalized incidence matrix $\mathbb{D}_{0\to q} \in \mathbb{R}^{n_{q} \times n_0}$ for $2 \leq q \leq n$, 
\[
\mathbb{D}_{0 \to q}(i,j) \hskip -2pt = \hskip -2pt \begin{cases}
	0 & \hskip -6pt \text{if $\sigma^0_j$ is not on the boundary of $\sigma^{q}_i$,} \\
	\Pi_{\sigma^1_k \cap \partial \sigma^q_i \neq \emptyset, \, \partial \sigma^1_k \cap \sigma^0_j   \neq \emptyset} \, \mathbb{D}_{0}(k,j) & \hskip -6pt \text{if $\sigma^0_j$ is on the boundary of $\sigma^{q}_i$,} \\
	%-1 & \text{if $\sigma^p_j$ is not coherent with the induced orientation of $\sigma^{q}_i$}.
\end{cases}
\]
and define $\mathbb{D}_{0\to 1} := \mathbb{D}_0$. Here we use the notation $\partial \sigma^q_i$ to denote the boundary of $\sigma^q_i$ which consists of all the $p$-cells, $0\leq p \leq q-1$, that are on the boundary of $\sigma^q_i$. In 3D, the edge-node, face-node, and volume-node incidence matrices, $\mathbb{D}_{0 \to 1}$, $\mathbb{D}_{0 \to 2}$, and $\mathbb{D}_{0 \to 3}$, are useful for our discretization.

\section{Fractional discrete exterior calculus via reformulation}\label{sec:FDEC}
This section describes the main result of this work, namely, the discretization of T-FVC via DEC. In Section~\ref{sec:equivalence}, we show an equivalent way to write Tarasov's FVC, which is suitable for discretization using DEC. Our discretization of T-FVC using DEC is then described in detail in Section~\ref{sec:discretization}.

\subsection{Reformulation and equivalence} \label{sec:equivalence}
In this section, we define a Caputo fractional exterior derivative of order $\alpha$ that is equivalent to T-FVC in 3D, and then reformulate this fractional exterior derivative into a form that can be easily discretized using DEC. Again, we mainly focus on the case $0<\alpha<1$ to keep the presentation simple and comment that our results hold for general $\alpha > 0$. To avoid confusion, we use $\beta>0$ to denote a fractional integration or differentiation order that is not necessarily $1-\alpha$ or $\alpha$. In addition, although we use exterior calculus notations mainly in this section, we confine our discussion on the region $\Omega = [x^1_{\min}, x^1_{\max}] \times \cdots \times [x^m_{\min}, x^m_{\max}] \subset \mathbb{R}^m$ and leave the general manifold case for future work.

In the following lemma, we present some properties of the fractional partial integrals and the partial Riemann-Liouville fractional derivatives introduced in Section~\ref{prelim:FC}. Those properties are useful for the reformulation.
\begin{lemma}
	Let $f : \Omega \subseteq \R^m \to \R$, $g : \Omega \subseteq \R^m \to \R$ be sufficiently smooth scalar fields. Then the following identities hold.
	\begin{align}
		\fracint{\beta}{}{x^j} (c_1 f + c_2 g) &= c_1 \fracint{\beta}{}{x^j} f + c_2 \fracint{\beta}{}{x^j} g \quad (\beta>0,c_1 \in \R, c_2 \in \R), \label{eqn:lem-id-3} \\
		\fracint{\beta_1}{}{x^i} \, \fracint{\beta_2}{}{x^j} f &= \fracint{\beta_2}{}{x^j} \, \fracint{\beta_1}{}{x^i} f \quad (\beta_1>0,\beta_2>0,i \neq j), \label{eqn:lem-id-4}\\
		D^k_{x^i} \, \fracint{\beta}{}{x^j} f &= \fracint{\beta}{}{x^j} \, D^k_{x^i} f \quad (\beta>0,k \in \mathbb{N}, i \neq j), \label{eqn:lem-id-5}\\
		\fracint{\beta_1}{}{x^i} \, \rlderiv{\beta_2}{}{x^j} f &=  \rlderiv{\beta_2}{}{x^j} \, \fracint{\beta_1}{}{x^i} f \quad (\beta_1 > 0, \beta_2 > 0, i \neq j). \label{eqn:lem-id-6}
	\end{align}
\end{lemma}
\begin{proof}
	\eqref{eqn:lem-id-3} can be easily verified by the linearity of the regular integral, which implies that the Riemann-Liouville fractional partial integral is linear. \eqref{eqn:lem-id-4} is a consequence of Fubini's theorem.

	To prove \eqref{eqn:lem-id-5}, we use mathematical induction on $k$. The base case $k=1$ can be shown using the Leibniz integral rule directly. For the inductive step, assume $D_{x^i}^k \fracint{\beta}{}{x^j} = \fracint{\beta}{}{x^j} D_{x^i}^k$ for some $k \in \mathbb{N}$. Then
	\[
	D_{x^i}^{k+1} \fracint{\beta}{}{x^j}
	= D_{x^i} D_{x^i}^k \fracint{\beta}{}{x^j}
	= D_{x^i} \fracint{\beta}{}{x^j} D_{x^i}^k
	= \fracint{\beta}{}{x^j} D_{x^i} D_{x^i}^k
	= \fracint{\beta}{}{x^j} D_{x^i}^{k+1}.
	\]
	Finally, \eqref{eqn:lem-id-6} follows from \eqref{eqn:lem-id-4} and  \eqref{eqn:lem-id-5}.
\end{proof}

Next, we define a Caputo fractional exterior derivative as follows:
\begin{definition} \label{def:Caputo-FVC-operators}
	We define the Caputo fractional exterior derivatives of order $\alpha > 0$, $d_p^\alpha: \Lambda^p \to \Lambda^{p+1}$, as
	\[
	d_p^\alpha \left( \sum_{J \in \mathfrak{J}_{p,m}} a_J dx^J \right) := \sum_{J \in \mathfrak{J}_{p,m}} \sum_{i=1}^m \left(\cderiv{\alpha}{}{x^i} a_J\right)\, dx^i \wedge dx^J.
	\]
	where coefficients $a_J$ are real-valued functions on $\Omega$. 
\end{definition}
Note that Definition~\ref{def:Caputo-FVC-operators} is essentially the same as the fractional exterior derivative that appears in \cite{baleanu2009fractional}. It is also similar to the fractional exterior derivative defined in \cite{cottrill2001fractional}, where the Riemann-Liouville, rather than Caputo, partial derivative is used.

Additionally, in 3D, by direct calculation and identifying $1$- and $2$-forms as vector fields and $3$-forms as scalar functions, we can verify that Definition~\ref{def:Caputo-FVC-operators} is equivalent to the T-FVC operators \eqref{def:T-FVC-operators}, which is summarized in the following proposition.
\begin{proposition} \label{prop:equivalence}
In 3D, for $\alpha > 0$, 
$d_0^\alpha = \operatorname{grad}^{\alpha}_T$, $d_1^\alpha = \operatorname{curl}^{\alpha}_T$, and $d_2^\alpha =  \operatorname{div}^{\alpha}_T$.
\end{proposition}

Although it is possible to show that $d_{p+1}^{\alpha} d_p^{\alpha} = 0$ directly from Definition~\ref{def:Caputo-FVC-operators}, here we take another approach by reformulating $d_p^{\alpha}$, which also allows us to discretize $d_p^{\alpha}$ using DEC intuitively. To this end, first, we need to define the following Riemann-Liouville fractional integration operator of order $\beta > 0$, $\fraciint{\beta}{p}: \Lambda^p \to \Lambda^p$, for $p \geq 1$,
\[
\fraciint{\beta}{p} \left( \sum_{J \in \mathfrak{J}_{p,m}} a_J dx^J \right) := \sum_{J = (i_1,\dots,i_p) \in \mathfrak{J}_{p,m}} \left(\fracint{\beta}{}{x^{i_1}} \cdots \fracint{\beta}{}{x^{i_p}} a_J\right) dx^J,
\]
and the following Riemann-Liouville fractional differentiation operator of order $\beta > 0$, $\rlderiiv{\beta}{p}:\Lambda^p \to \Lambda^p$, for $p \geq 1$,
\[
\rlderiiv{\beta}{p} \left( \sum_{J \in \mathfrak{J}_{p,m}} a_J dx^J \right) := \sum_{J = (i_1,\dots,i_p) \in \mathfrak{J}_{p,m}} \left(\rlderiv{\beta}{}{x^{i_1}} \cdots \rlderiv{\beta}{}{x^{i_p}} a_J\right) dx^J.
\]
For the $p=0$ case, we simply define $\fraciint{\beta}{0} := \operatorname{id}$ and $\rlderiiv{\beta}{0} := \operatorname{id}$. Note that the iterated fractional integrals and derivatives $\fracint{\beta}{}{x^{i_1}} \cdots \fracint{\beta}{}{x^{i_p}}$ and $\rlderiv{\beta}{}{x^{i_1}} \cdots \rlderiv{\beta}{}{x^{i_p}}$ present in these definitions have been defined in \cite{kilbas2006theory}, in which the iterated fractional integrals are called ``mixed Riemann-Liouville fractional integrals with respect to a part of the variables'' (and similarly for the iterated fractional derivative).

Using $\fraciint{1-\alpha}{p}$ and $\rlderiiv{1-\alpha}{p}$, we can reformulate $d_p^{\alpha}$ as follows. 
\begin{theorem} \label{thm:FDEC-alter}
	Let $0 < \alpha < 1$. Then
	$d_p^\alpha = \fraciint{1-\alpha}{p+1} \,d_p\, \rlderiiv{1-\alpha}{p}$.
\end{theorem}
\begin{proof}
	If $p=0$, then by definition, $\rlderiiv{1-\alpha}{0} = \operatorname{id}$, thus
	$
	\fraciint{1-\alpha}{p+1} \, d_p \, \rlderiiv{1-\alpha}{p} = \fraciint{1-\alpha}{1} \, d_0.
	$
	Taking a 0-form, i.e., a scalar field, $f : \R^m \to \R$, 
	we have,
	\[
	\fraciint{1-\alpha}{1}  \, d_0 f = \fraciint{1-\alpha}{1} \sum_{i=1}^m \left(D_{x^i} f\right) \, dx^i = \sum_{i=1}^m \left(\fracint{1-\alpha}{}{x^i} D_{x^i} f\right)\, dx^i = \sum_{i=1}^m \left(\cderiv{\alpha}{}{x^i} f\right)\, dx^i = d_0^\alpha f.
	\]
	If $p \geq 1$, then for a $p$-form
	$
	\omega^p = \sum_{J \in \mathfrak{J}_{p,m}} a_J dx^J,
	$
	we have
	\begin{align*}
		d_p \, \rlderiiv{1-\alpha}{p} \, \omega^p &= d_p \sum_{J \in \mathfrak{J}_{p,m}} \left(\rlderiv{1-\alpha}{}{x^{i_1}} \cdots \rlderiv{1-\alpha}{}{x^{i_p}} a_J \right) \, dx^J \\
		&= \sum_{J \in \mathfrak{J}_{p,m}} \sum_{i=1}^m \left(D_{x^i} \rlderiv{1-\alpha}{}{x^{i_1}} \cdots \rlderiv{1-\alpha}{}{x^{i_p}} a_J\right)\, dx^i \wedge dx^J \\
		&= \sum_{J \in \mathfrak{J}_{p,m}} \sum_{i \in \{1,\dots,m\} \setminus J} \left(D_{x^i} \rlderiv{1-\alpha}{}{x^{i_1}} \cdots \rlderiv{1-\alpha}{}{x^{i_p}} a_J\right)\, dx^i \wedge dx^J,
	\end{align*}
where we use the fact that $dx^i \wedge dx^J = 0$ if $i \in J$ in the last step.  Finally, applying $\fraciint{1-\alpha}{p+1}$ from the left, we obtain
	\begin{align*}
	& \qquad	\fraciint{1-\alpha}{p+1} \, d_p \, \rlderiiv{1-\alpha}{p} \omega^p\\ &= \sum_{J \in \mathfrak{J}_{p,m}} \sum_{i \in \{1,\dots,m\} \setminus J} \left(\fracint{1-\alpha}{}{x^i} \fracint{1-\alpha}{}{x^{i_1}} \cdots \fracint{1-\alpha}{}{x^{i_p}} D_{x^i} \rlderiv{1-\alpha}{}{x^{i_1}} \cdots \rlderiv{1-\alpha}{}{x^{i_p}} a_J\right) \, dx^i \wedge dx^J & \\ %\text{by \eqref{eqn:lem-id-3}, \eqref{eqn:lem-id-4}} \\
		&= \sum_{J \in \mathfrak{J}_{p,m}} \sum_{i=1}^m \left(\fracint{1-\alpha}{}{x^i} D_{x^i} \fracint{1-\alpha}{}{x^{i_1}} \cdots \fracint{1-\alpha}{}{x^{i_p}} \rlderiv{1-\alpha}{}{x^{i_1}} \cdots \rlderiv{1-\alpha}{}{x^{i_p}} a_J\right)\, dx^i \wedge dx^J & \\ %\text{by \eqref{eqn:lem-id-5}} \\
		&= \sum_{J \in \mathfrak{J}_{p,m}} \sum_{i=1}^m \left(\cderiv{\alpha}{}{x^i} \fracint{1-\alpha}{}{x^{i_1}} \rlderiv{1-\alpha}{}{x^{i_1}} \cdots \fracint{1-\alpha}{}{x^{i_p}} \rlderiv{1-\alpha}{}{x^{i_p}} a_J\right)\, dx^i \wedge dx^J & \\ %\text{by \eqref{eqn:lem-id-6}} \\
		&= \sum_{J \in \mathfrak{J}_{p,m}} \sum_{i=1}^m \left(\cderiv{\alpha}{}{x^i} a_J\right)\, dx^i \wedge dx^J &  \\
		%\text{by Lemma~\ref{lem:IRLD}} \\
		&= d_p^\alpha \omega^p.
	\end{align*}
This completes the proof. 
\end{proof}
\begin{remark}
Following the same argument, Theorem~\ref{thm:FDEC-alter} can be generalized to any $ 0< \alpha \notin \mathbb{N}$ by using high-order analogues of the exterior derivatives.
%In this case, we have $d_p^\alpha = \fraciint{n-\alpha}{p+1} \, d_p^n \, \rlderiiv{n-\alpha}{p}$, where $n = \lfloor \alpha \rfloor + 1$. 
\end{remark}

From Lemma~\ref{lem:FTFC} and \eqref{eqn:lem-id-6}, we can easily verify $\rlderiiv{1-\alpha}{p} \, \fraciint{1-\alpha}{p} = \operatorname{id}$. Thus,
\begin{equation*}
	d_{p+1}^{\alpha} \, d_p^{\alpha} = \left(\fraciint{1-\alpha}{p+2} \, d_{p+1} \, \rlderiiv{1-\alpha}{p+1}\right) \left(\fraciint{1-\alpha}{p+1} \, d_p \, \rlderiiv{1-\alpha}{p}\right) = \fraciint{1-\alpha}{p+2} \, d_{p+1}  d_p \rlderiiv{1-\alpha}{p} = 0.
\end{equation*}

From Proposition~\ref{prop:equivalence} and Theorem~\ref{thm:FDEC-alter}, we have the following corollary which reformulates the T-FVC operators. Such a reformulation enables us to discretize T-FVC using the DEC framework.  
\begin{corollary}
In 3D, for $0 < \alpha <1$, we have
\begin{equation*}
	\operatorname{grad}_T^{\alpha} = \fraciint{1-\alpha}{1} \, \operatorname{grad}, \ \ \operatorname{curl}_T^{\alpha} = \fraciint{1-\alpha}{2} \, \operatorname{curl} \, \rlderiiv{1-\alpha}{1}, \ \  \operatorname{div}_T^{\alpha} = \fraciint{1-\alpha}{3} \, \operatorname{div} \, \rlderiiv{1-\alpha}{2}.
\end{equation*}
\end{corollary}

\subsection{Definition of FDEC by discretization} \label{sec:discretization}
In this section, we discretize the fractional exterior derivatives $d_p^\alpha$ in 3D to produce the corresponding FDEC operators $\mathbb{D}_p^\alpha$, $p=0,1,2$. Our discrete exterior derivative is defined as the composition of the discretizations of each of the three composite operators in our reformulated fractional exterior derivative, $d_p^\alpha = \fraciint{1-\alpha}{p+1} \, d_p \, \rlderiiv{1-\alpha}{p}$.
First, as explained in Section~\ref{prelim:DEC}, the discretization of $d_p$ is $\mathbb{D}_p$. Next we need to discretize $\fraciint{\beta}{p} : \Lambda^p \to \Lambda^p $ and $\rlderiiv{\beta}{p} : \Lambda^p \to \Lambda^p$. Naturally, their discretizations should map $p$-cochains to $p$-cochains, i.e., matrices of size $n_p \times n_p$. Furthermore, noting that $\rlderiiv{\beta}{p} \, \fraciint{\beta}{p} = \operatorname{id}$, we use a matrix $\dfraciint{\beta}{p} \in \mathbb{R}^{n_p \times n_p}: C^p \to C^p$ (which will be defined later) as the discretization of $\fraciint{\beta}{p}$, while $\rlderiiv{\beta}{p} $ is discretized by $\left(\dfraciint{\beta}{p}\right)^{-1} \in \mathbb{R}^{n_p \times n_p}: C^p \to C^p$ in order to enforce the structure-preserving property. The resultant discrete exterior derivative is hence defined as $\mathbb{D}_p^\alpha = \dfraciint{1-\alpha}{p+1} \, \mathbb{D}_p \, (\dfraciint{1-\alpha}{p})^{-1}$.

Our discretization was done on a 3D regular cubical complex as shown in Figure~\ref{fig:3D-cubic}. It is obtained by dividing the parallelepiped that the discretized $d_p^\alpha$ operators are defined on, $\Omega = [x_{\min},x_{\max}] \times [y_{\min},y_{\max}] \times [z_{\min},z_{\max}]$, into cuboids. Specifically, the division is $x_{\min} = x_1 < x_2 \cdots < x_{n_x+1} = x_{\max}$, $y_{\min} = y_1 < y_2 \cdots < y_{n_y+1} = y_{\max}$, and $z_{\min} = z_1 < z_2 \cdots < z_{n_z+1} = z_{\max}$, respectively (see Figure~\ref{fig:3D-cubic}).

\begin{figure}[htbp]
	\centering
\includegraphics[width=0.6\linewidth]{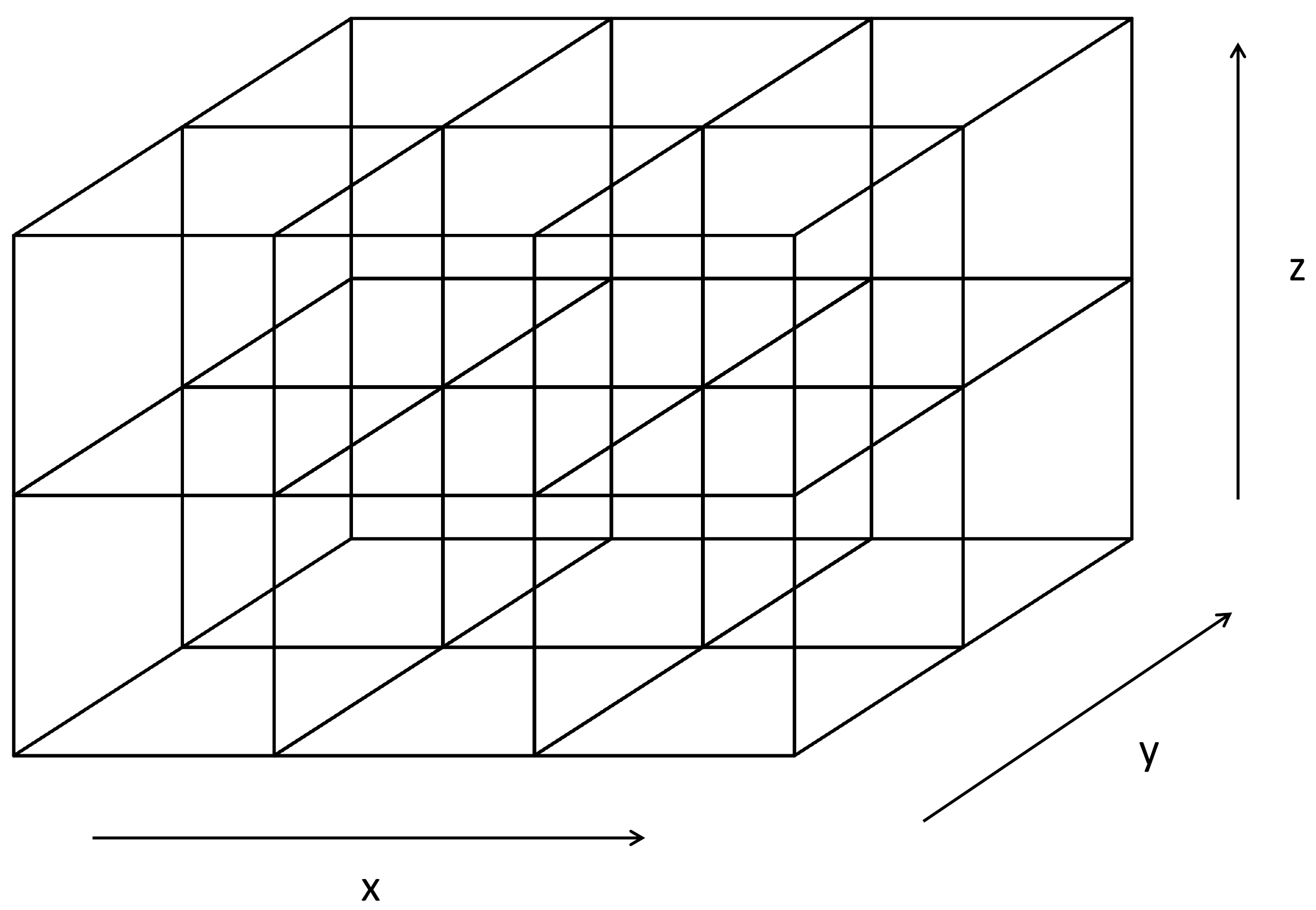}
	\caption{Diagram of a 3D regular cubical complex, with $n_x=3$, $n_y=2$, $n_z=2$}\label{fig:3D-cubic}
\end{figure}
We only present the discretization on the cubical complex because, while constructing the $\mathbb{D}_p$ matrices is straightforward for general meshes, defining the $\dfraciint{\beta}{p}$ matrices on a general mesh is difficult and challenging for practical implementation. 
%For a general mesh, they can be defined by $\dfraciint{\beta}{p} = \mathcal{R}_p \, \fraciint{\beta}{p} \, \mathcal{R}^{-1}_p$ where $\mathcal{R}^{-1}_p : C^p \to \Lambda^p$ is an interpolation map such as the Whitney map. However, 
In particular, computing the entries in this matrix requires computing multiple fractional integrals through multiple mesh elements, which makes efficient implementation difficult in practice although it is theoretically feasible. This computation is much easier, however, when a square or cube grid is used. We would like to point out that, such a difficulty also arises in other types of discretizations for fractional derivatives and, to the best of our knowledge, finding an efficient method for computing those fractional integrals on general meshes is still an open question in the community. 
%, which is largely due to the fact that the T-FVC is standard basis directional.

We will first discuss the strategy of discretizing $\fraciint{\beta}{p}$ on the cubical complex. We will then introduce two FDEC operators: the fractional discrete exterior derivative $\mathbb{D}_p^\alpha$, $p=0,1,2$, and the fractional de Rham map, $\mathcal{R}_p^\alpha$, $p=0,1,2,3$, and discuss their properties. 

\subsubsection{Discretization of $\fraciint{\beta}{p}$}
Next, we consider a discretization of the fractional integral operator $\fraciint{\beta}{p}$, $p=1,2,3$. We use $\fraciint{\beta}{1}$ as an example to illustrate the general procedure for discretization of $\fraciint{\beta}{p}$. 

On the cubical complex, there are $n_{1,x}$, $n_{1,y}$, and $n_{1,z}$ ($n_{1,x} + n_{1,y} + n_{1,z} = n_1$) edges are along the $x$, $y$, and $z$ direction, respectively. Considering an $x$ direction edge, $\sigma^{1,x}_{i,j,k} = [x_i,x_{i+1}] \times y_j \times z_k$, for any 1-form $\omega^1 = f_x dx + f_y dy + f_z dz$, we apply $\fraciint{\beta}{1}$ and take the integral along $\sigma^{1,x}_{i,j,k}$ to obtain,
\begin{align}
\int_{\sigma^{1,x}_{i,j,k}} \fraciint{\beta}{1} \omega^1 &= \fracint{1}{x_i}{x_{i+1}}[x'] \, \fracint{\beta}{x_1}{x',x}[x''] f_x \nonumber \\[-10pt] 
	&= \fracint{1}{x_1}{x_{i+1}}[x'] \, \fracint{\beta}{x_1}{x',x}[x''] f_x
	- \fracint{1}{x_1}{x_i}[x'] \, \fracint{\beta}{x_1}{x',x}[x''] f_x \nonumber \\
	&= \fracint{1+\beta}{x_1}{x_{i+1},x}[x'] f_x
	- \fracint{1+\beta}{x_1}{x_i,x}[x'] f_x. \label{eqn:int_I_1^beta}
\end{align}
In the last step, we used fractional integrals' semigroup property (Lemma~\ref{lem:semigroup}). \eqref{eqn:int_I_1^beta} implies that the above integral can be computed by a signed sum of the values of the $x$ component of $\fraciint{1+\beta}{1} \omega^1$ on the edge's two incident nodes. If we write the edge-node incidence matrix $\mathbb{D}_{0\to1} = \mathbb{D}_0$ into a block form according to edges along each direction, i.e., 
\begin{equation} \label{eqn:D0-block-form}
\mathbb{D}_0 = \begin{pmatrix}
	\mathbb{B}_{1,x}\\ 
	\mathbb{B}_{1,y} \\
	\mathbb{B}_{1,z}
\end{pmatrix}.
\end{equation}
Then such a signed sum can be encoded into the matrix $\mathbb{B}_{1,x}$.

The main step of our discretization is to compute $\fracint{1+\beta}{x_1}{x_i,x}[x'] f_x$ approximately at the node $\sigma^0_{i,j,k}$ by using a piecewise constant approximation on each edge. Specifically, on edge $\sigma^{1,x}_{i',j,k}$, we use the average $\overline{\omega^1} := |\sigma^{1,x}_{i',j,k}|^{-1} \int_{\sigma^{1,x}_{i',j,k}} \omega^1 = |\sigma^{1,x}_{i',j,k}|^{-1}\mathbf{c}^1(\sigma^{1,x}_{i',j,k})$, where $\mathbf{c}^1 = \mathcal{R}_1 \omega^1$, to be the constant approximation of $\omega^1$, and then have
\begin{align*}
	\fracint{1+\beta}{x_1}{x_i,x}[x'] f_x &= \frac{1}{\Gamma(1+\beta)} \int_{x_1}^{x_i} \frac{1}{(x_i - x')^{-\beta}} f_x \, dx' \\
	& = \sum_{i'=1}^{i-1} \frac{1}{\Gamma(1+\beta)} \int_{x_{i'}}^{x_{i'+1}} \frac{1}{(x_{i} - x')^{-\beta}}f_x \, dx' \\
	& \approx \sum_{i'=1}^{i-1}  \frac{1}{\Gamma(1+\beta)} \int_{x_{i'}}^{x_{i'+1}} \frac{1}{(x_{i} - x')^{-\beta}} \overline{\omega^1} \, dx \\
	& = \sum_{i'=1}^{i-1} \frac{\mathbf{c}^1(\sigma^{1,x}_{i',j,k})}{|\sigma^{1,x}_{i',j,k}|} \frac{(x_i - x_{i'})^{1+\beta} - (x_i - x_{i'+1})^{1+\beta}}{\Gamma(2+\beta)}.
\end{align*}
Thus, if we let $C^1_x$ be the space of $1$-cochains that uses $1$-cells along the $x$ direction only, then the matrix form of the discrete Riemann-Liouville fractional integral operator of the $x$ component, $\dfraciint{\beta}{1,x} \in \mathbb{R}^{n_{1,x} \times n_{1,x}}: C^1_x \to C^1_x$, is defined as 
\begin{equation*}
\dfraciint{\beta}{1,x} : = \mathbb{B}_{1,x} \mathbb{M}_{1,x}^{1+\beta} \mathbb{V}_{1,x}^{-1}.
\end{equation*}
Here $\mathbb{B}_{1,x} \in \mathbb{R}^{n_{1,x}\times n_0}$ is defined in~\eqref{eqn:D0-block-form}, $\mathbb{V}_{1,x} = \operatorname{diag}(|\sigma^{1,x}_{i,j,k}|) \in \mathbb{R}^{n_{1,x} \times n_{1,x}}$, and $\mathbb{M}_{1,x}^{1+\beta} := M_{1,x}^{1+\beta} \otimes I_{n_y+1} \otimes I_{n_z+1} \in \mathbb{R}^{n_0 \times n_{1,x}}$ 
where $\otimes$ denotes the standard Kronecker product, $I_{n}$ denotes a $n \times n$ identity matrix, and $M_{1,x}^{1+\beta} \in \mathbb{R}^{(n_x+1) \times n_x}$ is defined as follows,
\begin{equation*}
\left(M_{1,x}^{1+\beta}\right)_{i,i'} = 
\begin{cases}
		\displaystyle
	\frac{(x_i - x_{i'})^{1+\beta} - (x_i - x_{i'+1})^{1+\beta}}{\Gamma(2+\beta)}, & \text{if} \ i' < i, \\
	0,  &\text{otherwise}.
\end{cases}
\end{equation*} 
Repeating the same procedure for a $y$ direction edge, $\sigma^{1,y}_{i,j,k} = x_i \times [y_j, y_{j+1}] \times z_k$, and a $z$ direction edge, $\sigma^{1,z}_{i,j,k} = x_i \times y_j \times [z_k, z_{k+1}]$, we can define the discrete Riemann-Liouville fractional integral operator of the $y$ and $z$ component, $\dfraciint{\beta}{1,y} $ and $\dfraciint{\beta}{1,z} $, as follows,
\begin{equation*}
\dfraciint{\beta}{1,y} : = \mathbb{B}_{1,y} \, \mathbb{M}_{1,y}^{1+\beta} \, \mathbb{V}_{1,y}^{-1} \quad \text{and} \quad \dfraciint{\beta}{1,z} : = \mathbb{B}_{1,z} \, \mathbb{M}_{1,z}^{1+\beta} \, \mathbb{V}_{1,z}^{-1}.	
\end{equation*}
where $\mathbb{B}_{1,y} \in \mathbb{R}^{n_{1,y}\times n_0}$ and $\mathbb{B}_{1,z} \in \mathbb{R}^{n_{1,z}\times n_0}$ are defined in \eqref{eqn:D0-block-form}, $\mathbb{V}_{1,y} = \operatorname{diag}(|\sigma^{1,y}_{i,j,k}|) \in \mathbb{R}^{n_{1,y} \times n_{1,y}}$, and $\mathbb{V}_{1,z} = \operatorname{diag}(|\sigma^{1,z}_{i,j,k}|) \in \mathbb{R}^{n_{1,z} \times n_{1,z}}$. We also similarly define $\mathbb{M}_{1,y}^{1+\beta} :=  I_{n_x+1} \otimes  M_{1,y}^{1+\beta} \otimes I_{n_z+1} \in \mathbb{R}^{n_0 \times n_{1,y}}$ where
\begin{equation*}
	\left(M_{1,y}^{1+\beta}\right)_{j,j'} = 
	\begin{cases}
		\displaystyle
		\frac{(y_j - y_{j'})^{1+\beta} - (y_j - y_{j'+1})^{1+\beta}}{\Gamma(2+\beta)}, & \text{if} \ j' < j, \\
		0,  &\text{otherwise},
	\end{cases}
\end{equation*} 
and $\mathbb{M}_{1,z}^{1+\beta} := I_{n_x+1} \otimes I_{n_y+1}\otimes M_{1,z}^{1+\beta}   \in \mathbb{R}^{n_0 \times n_{1,z}}$ where
\begin{equation*}
	\left(M_{1,z}^{1+\beta}\right)_{k,k'} = 
	\begin{cases}
		\displaystyle
		\frac{(z_k - z_{k'})^{1+\beta} - (z_k - z_{k'+1})^{1+\beta}}{\Gamma(2+\beta)}, & \text{if} \ k' < k, \\
		0,  &\text{otherwise}.
	\end{cases}
\end{equation*} 
Finally, combining $\dfraciint{\beta}{1,x}$, $\dfraciint{\beta}{1,y}$, and $\dfraciint{\beta}{1,z}$, we define the overall discrete Riemann-Liouville fractional 1-form integral operator $\dfraciint{\beta}{1} \in \mathbb{R}^{n_1 \times n_1}: C^1 \to C^1$ as follows,
\begin{equation*}
	\dfraciint{\beta}{1} := \mathbb{B}_1 \, \mathbb{M}_1^{1+\beta} \, \mathbb{V}_1^{-1} = \operatorname{diag}(\dfraciint{\beta}{1,x}, \dfraciint{\beta}{1,y}, \dfraciint{\beta}{1,z}),
\end{equation*}
where $\mathbb{B}_{1} = \operatorname{diag}(\mathbb{B}_{1,x}, \mathbb{B}_{1,y}, \mathbb{B}_{1,z})$, $\mathbb{M}_1^{1+\beta} = \operatorname{diag}(\mathbb{M}_{1,x}^{1+\beta},\mathbb{M}_{1,y}^{1+\beta},\mathbb{M}_{1,z}^{1+\beta})$, and \\
$\mathbb{V}_1 = \operatorname{diag}(\mathbb{V}_{1,x}, \mathbb{V}_{1,y}, \mathbb{V}_{1,z})$.

Analogously, to discretize $\mathcal{I}_2^\beta$, we notice that there are $n_{2,yz}$ faces parallel to the $yz$ plane, $n_{2,xz}$ faces parallel to the $xz$ plane, and $n_{2,xy}$ faces parallel to the $xy$ plane. So we expect the discrete fractional integral $\dfraciint{\beta}{2}$ to also have a diagonal block form. To discretize $\fraciint{\beta}{2}$, for any $2$-form $\omega^2$, we apply $\fraciint{\beta}{2}$ and then take integration on a face. Using the average value of $\omega^2$ on the face as a constant approximation, we can compute approximately the face integral, leading to the discrete fractional 2-form integral operator. We omit the details here and directly present $\dfraciint{\beta}{2}$. First, we need the signed face-node incidence matrix $\mathbb{D}_{0 \to 2} \in \mathbb{R}^{n_2 \times n_0}$, which has the block form  
$$
\mathbb{D}_{0 \to 2} = 
\begin{pmatrix} 
\mathbb{B}_{2,yz} \\
\mathbb{B}_{2,xz} \\
\mathbb{B}_{2,xy}
\end{pmatrix}.
$$
Then we need matrices $\mathbb{M}_{2,yz}^{1+\beta}  \in \mathbb{R}^{n_0 \times n_{2,yz}}, \mathbb{M}_{2,xz}^{1+\beta}  \in \mathbb{R}^{n_0 \times n_{2,xz}}, \mathbb{M}_{2,xy}^{1+\beta} \in \mathbb{R}^{n_0 \times n_{2,xy}}$, which are the discretizations of approximately evaluating the $dy \wedge dz$, $dz \wedge dx$, and $dx \wedge dy$ components of $\fraciint{1+\beta}{2} \omega^2$ at each node, respectively. The overall discrete Riemann-Liouville fractional integral operator $\dfraciint{\beta}{2} \in \mathbb{R}^{n_2 \times n_2}: C^2 \to C^2$ is defined as
\begin{equation*}
	\dfraciint{\beta}{2}	:= \mathbb{B}_2 \, \mathbb{M}_2^{1+\beta} \, \mathbb{V}_2^{-1},
\end{equation*}
where $\mathbb{B}_{2} = \operatorname{diag}(\mathbb{B}_{2,yz}, \mathbb{B}_{2,xz}, \mathbb{B}_{2,xy})$, $\mathbb{M}_2^{1+\beta} = \operatorname{diag}(\mathbb{M}_{2,yz}^{1+\beta},\mathbb{M}_{2,xz}^{1+\beta},\mathbb{M}_{2,xy}^{1+\beta})$, and \\
$\mathbb{V}_2 = \operatorname{diag}(\operatorname{diag}(|\sigma^{2,yz}_{i,j,k}|),\operatorname{diag}(|\sigma^{2,xz}_{i,j,k}|),\operatorname{diag}(|\sigma^{2,xy}_{i,j,k}|))$.

Finally, to discretize $\mathcal{I}_3^\beta$, note that there is only one type of volume (cuboids) in the cubical complex. We no longer expect the discrete fractional integral $\dfraciint{\beta}{3}$ to have a diagonal block form. Similarly, to discretize $\fraciint{\beta}{3}$, for any $3$-form $\omega^3$, we apply $\fraciint{\beta}{3}$ and then take a volume integration. Using the average value of $\omega^3$ on the volume as a constant approximation, we can compute approximately the volume integral, leading to the discrete fractional integral on 3-forms operator. Letting $\mathbb{M}_3^{1+\beta} \in \mathbb{R}^{n_0 \times n_3}$ be the discretization of evaluating $\fraciint{1+\beta}{3} \omega^3$ at each node, we define the discrete Riemann-Liouville fractional integral operator $\dfraciint{\beta}{3} \in \mathbb{R}^{n_3 \times n_3}: C^3 \to C^3$ as
\begin{equation*}
	\dfraciint{\beta}{3}	:= \mathbb{B}_3 \, \mathbb{M}_3^{1+\beta} \, \mathbb{V}_3^{-1},
\end{equation*}
where $\mathbb{B}_3:= \mathbb{D}_{0\to3} \in \mathbb{R}^{n_3 \times n_0}$ is the signed volume-node incidence matrix and $\mathbb{V}_3 = \operatorname{diag}(|\sigma^3_{i,j,k}|)$.

In general, our discrete Riemann-Liouville fractional integral operators are defined as 
\begin{equation*}
	\dfraciint{\beta}{p} = \mathbb{B}_p \, \mathbb{M}_p^{1+\beta} \, \mathbb{V}_p^{-1}, \quad p =1, 2, 3.
\end{equation*}
We summarize the sizes and number of nonzeros of $\mathbb{B}_p$ and $\mathbb{M}_p^{1+\beta}$ as a function of $n$, $n = n_x = n_y = n_z$, in Table~\ref{table:sparsity-B-M}. We exclude $\mathbb{V}_p$ because they are diagonal. As we can see, on the cubical complex, the matrices $\mathbb{B}_p$ are sparse since they are constructed from incidence matrices. On the other hand, the matrices $\mathbb{M}_p^{1+\beta}$ become denser and denser as $p$ increases. However, they are still relatively sparse. Even for $p=3$, only $12.5\%$ of the entries are nonzeros. 

\begin{table}
	\centering
	\caption{Sizes and number of nonzeros of the $\mathbb{B}_p$ and $\mathbb{M}_p^{1+\beta}$ matrices as function of $n$, $n = n_x =n_y=n_z$.
		($n_0 = (n+1)^3$, $n_1 = 3n(n+1)^2$, $n_2 = 3n^2(n+1)$, and $n_3 = n^3$)}
	\label{table:sparsity-B-M}
	\begin{tabular}{lllll} \toprule
		& \multicolumn{2}{c}{size}              & \multicolumn{2}{c}{number of nonzeros}                                     \\ \cmidrule(r){2-3} \cmidrule(r){4-5}
		$p$ & $\mathbb{B}_p$             & $\mathbb{M}_p^{1+\beta}$       & $\mathbb{B}_p$  & $\mathbb{M}_p^{1+\beta}$                                            \\ \midrule
		1   & $n_1 \times 3n_0$ & $3n_0 \times n_1$ & $2n_1$ & $\frac{3}{2} n(n+1)^3$            \\[3pt]
		2   & $n_2 \times 3n_0$ & $3n_0 \times n_2$ & $4n_2$ & $\frac{3}{4} n^2(n+1)^3 $   \\[3pt]
		3   & $n_3 \times n_0$  & $n_0 \times n_3$  & $8n_3$ & $\frac{1}{8} n^3 (n+1)^3$                  
		\\ \bottomrule
	\end{tabular}
\end{table}

\subsubsection{FDEC operators}
After discretizing $\fraciint{\beta}{p}$, we now introduce two FDEC operators: the fractional discrete exterior derivative, and the fractional de Rham map. We first introduce the fractional discrete exterior derivative. As mentioned before, based on $\rlderiiv{\beta}{p} \, \fraciint{\beta}{p} = \operatorname{id}$ on the continuous level, we simply define the discrete version of the $\rlderiiv{\beta}{p}$ operator as the inverse of the $\dfraciint{\beta}{p}$ matrix to preserve the property on the discrete level. Based on the discrete versions of $\fraciint{\beta}{p}$, $\rlderiiv{\beta}{p}$, and $d_p$, i.e., $\dfraciint{\beta}{p}$, $\left( \dfraciint{\beta}{p} \right)^{-1}$, and $\mathbb{D}_p$, it is straightforward to define the FDEC operators $\mathbb{D}_p^\alpha$ as follows.
\begin{definition}\label{def:FDEC-operators}
	The fractional discrete exterior derivative operators are defined as
	\[ \mathbb{D}_p^\alpha := \dfraciint{1-\alpha}{p+1} \, \mathbb{D}_p \, (\dfraciint{1-\alpha}{p})^{-1}, \quad p=0,1,2, \ 0<\alpha<1,
	\]
where we define $\dfraciint{1-\alpha}{0} := I_{n_0}$, i.e., an identify matrix of size $n_0 \times n_0$.
\end{definition}

From Definition~\ref{def:FDEC-operators}, we can easily see that $\mathbb{D}_{p+1}^\alpha \mathbb{D}_p^\alpha = 0$ because 
\begin{align*}
\mathbb{D}_{p+1}^\alpha \, \mathbb{D}_p^\alpha = \dfraciint{1-\alpha}{p+2} \,  \mathbb{D}_{p+1} \, (\dfraciint{1-\alpha}{p+1} )^{-1}  \, \dfraciint{1-\alpha}{p+1}  \mathbb{D}_p \,  (\dfraciint{1-\alpha}{p})^{-1} = \dfraciint{1-\alpha}{p+2} \,  \mathbb{D}_{p+1} \, \mathbb{D}_p \, (\dfraciint{1-\alpha}{p})^{-1} = 0. \label{eqn:FDEC-exact}
\end{align*}

\begin{remark}
We can discretize $\rlderiiv{\beta}{p}$ directly using the same method as was done to discretize $\fraciint{\beta}{p}$ and obtain a discrete Riemann-Liouville fractional derivative operator,
\[
\drlderiiv{\beta}{p} = \mathbb{B}_p \, \mathbb{M}_p^{1-\beta} \, \left(\mathbb{V}_p \right)^{-1}, \quad p=0,1, \quad 0<\beta<1.
\]
This provides two alternative ways to define the FDEC operators. The first one is 
\begin{equation*}
	\mathbb{D}_p^\alpha : = \dfraciint{1-\alpha}{p+1} \, \mathbb{D}_p \, \drlderiiv{1-\alpha}{p}.
\end{equation*}
Although this approach seems to be more natural than Definition~\ref{def:FDEC-operators}, unfortunately,  these operators do not satisfy $\mathbb{D}_{p+1}^\alpha \, \mathbb{D}_p^\alpha = 0$ due to the fact that $\drlderiiv{\beta}{p} \, \dfraciint{\beta}{p} \neq I_{n_p}$. But $\drlderiiv{\beta}{p} \, \dfraciint{\beta}{p} \approx I_{n_p}$ as the mesh size gets smaller, which leads to $\mathbb{D}_{p+1}^\alpha \, \mathbb{D}_p^\alpha \approx 0$ when the mesh is refined. The second approach is 
\begin{equation*}
	\mathbb{D}_p^\alpha : = \left( \drlderiiv{1-\alpha}{p+1} \right)^{-1} \, \mathbb{D}_p \, \drlderiiv{1-\alpha}{p},
\end{equation*}
which also satisfies $\mathbb{D}_{p+1}^\alpha \, \mathbb{D}_p^\alpha = 0$.  However, we empirically observe worse convergence than the version we are using. This is why we decided to use Definition~\ref{def:FDEC-operators}. We would like to thoroughly understand these three approaches' approximation properties and convergence behaviors in our future work.  
\end{remark}

\begin{remark}
We comment that the FDEC operators defined in Definition~\ref{def:FDEC-operators} are closely related to the mimetic finite difference (MFD) method \cite{vabishchevich2005finite,da2014mimetic} since $\mathbb{D}_p^{\alpha} = \dfraciint{1-\alpha}{p+1} \, \mathbb{D}_p \, (\dfraciint{1-\alpha}{p})^{-1} = \left( \mathbb{B}_{p+1} \, \mathbb{M}_{p+1}^{2-\alpha} \right) \, \left( \mathbb{V}_{p+1}^{-1} \,  \mathbb{D}_p \, \mathbb{V}_p \right) \, \left( \mathbb{B}_p \mathbb{M}_p^{2-\alpha} \right)^{-1}$ and $\mathbb{V}_{p+1}^{-1} \,  \mathbb{D}_p \, \mathbb{V}_p$ are the MFD operators as pointed out in \cite{rodrigo2015finite,adler2021finite}.
\end{remark}

Next, after defining the fractional discrete exterior derivative, we define a ``fractional de Rham map,'' as follows. Recall that in the integer case, we have the identity $
\mathbb{D}_p \, \mathcal{R}_p = \mathcal{R}_{p+1} \, d_p$. However, in the fractional case, $\mathbb{D}_p^\alpha \mathcal{R}_p \neq \mathcal{R}_{p+1} d_p^\alpha$ in general. To remedy this, we define the following \emph{fractional de Rham map} $\mathcal{R}_p^\alpha: \Lambda^p \to C^p$:
\[
\mathcal{R}_p^\alpha := \dfraciint{1-\alpha}{p} \, \mathcal{R}_p \, \rlderiiv{1-\alpha}{p}, \quad p = 0, 1, 2, 3.
\]
Then $\mathbb{D}_p^\alpha \, \mathcal{R}_p^\alpha = \mathcal{R}_{p+1}^\alpha \, d_p^\alpha$ holds by the following direct calculation:
\begin{align*}
	&\mathbb{D}_p^\alpha \, \mathcal{R}_p^\alpha = \dfraciint{1-\alpha}{p+1} \, \mathbb{D}_p \, (\dfraciint{1-\alpha}{p})^{-1} \, \dfraciint{1-\alpha}{p} \, \mathcal{R}_p \, \rlderiiv{1-\alpha}{p} = \dfraciint{1-\alpha}{p+1} \, \mathbb{D}_p \, \mathcal{R}_p \, \rlderiiv{1-\alpha}{p} \nonumber \\ 
	& \quad = \dfraciint{1-\alpha}{p+1} \, \mathcal{R}_{p+1} \, d_p \, \rlderiiv{1-\alpha}{p} = \dfraciint{1-\alpha}{p+1} \, \mathcal{R}_{p+1} \, \rlderiiv{1-\alpha}{p+1} \, \fraciint{1-\alpha}{p+1} \, d_p \, \rlderiiv{1-\alpha}{p} = \mathcal{R}_{p+1}^\alpha \, d_p^\alpha.  \label{eqn:FDEC-FdeRham-commute}
\end{align*}
This implies the following commuting diagram
\begin{center}
	\begin{tikzcd}
		\Lambda^p \arrow[r, "d^{\alpha}_p"] \arrow[d, "\mathcal{R}^{\alpha}_p"] &  \Lambda^{p+1} \arrow[d, "\mathcal{R}^{\alpha}_{p+1}"] \\
		C^p \arrow[r, "\mathbb{D}^{\alpha}_p"]           &  C^{p+1}
	\end{tikzcd}
\end{center}
and suggests that, with the fractional de Rham map, the FDEC operator $\mathbb{D}_p^{\alpha}$ can be viewed as an ``error-free'' discretization of $d_p^{\alpha}$, which provides another structure-preserving property and, therefore, makes the FDEC operators useful for preserving physics laws in numerical simulations, e.g., fractional conservation of mass \cite{wheatcraft2008fractional} and fractional Gauss's laws \cite{tarasov2008fractional}.

\section{Numerical experiments} \label{sec:numerics}

In this section, we numerically test the FDEC operators. The MATLAB and Mathematica code is available at \url{https://github.com/71c/Frac-DEC}. In Section~\ref{sec:convergence}, we numerically study the convergence of $\mathbb{D}_p^\alpha$ to $d_p^\alpha$, while in Section~\ref{sec:dd0}, the property $\mathbb{D}_{p+1}^\alpha \mathbb{D}_p^\alpha = 0$ is verified numerically.

\subsection{Convergence of our FDEC to T-FVC} \label{sec:convergence}
In this section, we numerically study the approximation property of the FDEC operators (Definition~\ref{def:FDEC-operators}) to their continuous counterparts (Definition~\ref{def:Caputo-FVC-operators}) for $p=0,1,2$. Their convergence rates are tested using a regular cubical complex (see Figure~\ref{fig:3D-cubic}).

Recall that $\mathbb{D}_p^\alpha \mathcal{R}_p \neq \mathcal{R}_{p+1} d_p^\alpha$ in general, however, we do expect that $\mathbb{D}_p^\alpha \mathcal{R}_p \to \mathcal{R}_{p+1} d_p^\alpha$ as the number of the subdivisions increases (i.e., the mesh size decreases). Therefore, to properly check the convergence rate, for a given $p$-form $\omega^p \in \Lambda^p$,
we compute the following root mean square (RMS) error, % $\ell^2$ error
\begin{equation} \label{errorDpalpha}
	\operatorname{RMS}\left(\mathbb{V}_{p+1}^{-1} \mathbb{D}_p^\alpha \mathcal{R}_p \omega^p - \mathbb{V}_{p+1}^{-1} \mathcal{R}_{p+1} d_p^{\alpha} \omega^p\right),
\end{equation}
where $\operatorname{RMS}(\bm{x}) := \sqrt{\frac{1}{n}\sum_{i=1}^n x_i^2}$ for $\bm{x} \in \R^n$. Here $\mathbb{V}_{p+1}^{-1}$ makes sure that we are measuring the average error throughout space by dividing the cochain value on a $(p+1)$-cell by the volume of that $(p+1)$-cell.

% In Figures \ref{fig:fgrad_h}--\ref{fig:fdiv_alpha}, we plot the error (\ref{errorDpalpha}) against the number of subdivisions and the fractional order $\alpha$, for $p=0,1,2$.

In Figures~\ref{fig:convergence_h} and \ref{fig:convergence_alpha},
we plot the error (\ref{errorDpalpha}) against the number of subdivisions and the fractional order $\alpha$, for $p=0,1,2$. We use the region $\Omega=[0,1] \times [0,1] \times [0,1]$ for all plots. We use the scalar field
\[
f(x,y,z) = -8 x y^2 z-3 \cos \left(20 \left(x-1/2\right) (y-1) z\right)+4 \left(x-1/2\right)^2+\left(y-1/2\right)^2
\]
as $\omega^0$ to test the fractional gradient, and use the vector field
\[
\bm{F}(x, y, z) =
\begin{bmatrix}[1.3]
y \sin (5 x y+z)+3 \left(x z-\frac{1}{2}\right)^2-3 \left(y-\frac{1}{2}\right)^2 \\
z \cos (10 x y z)+x z-y^3 \\
2 \sin \left(5 x^3 y\right)+x y \left(z-\frac{1}{4}\right)+\cos (2 x y z)-x+y^3 z
\end{bmatrix}
\]
as $\omega^1$ and $\omega^2$ to test the fractional curl and divergence, respectively.

Figure~\ref{fig:convergence_h} plots the error (\ref{errorDpalpha}) for $p=0,1,2$ against the number of subdivisions for two values of $\alpha$: $\alpha=0.25$ and $\alpha=0.9$. From the plots, we can see that the convergence is generally slower at small $n$, and the convergence becomes faster approaching second-order convergence for large $n$. Thus, we conjecture that our proposed FDEC operators converge in second-order asymptotically. We can also see, at least for the fractional curl and divergence, that for $\alpha = 0.25$, the convergence is close to second-order, while for $\alpha = 0.9$, the convergence is slower than second-order (at least empirically for small $n$).  This likely means that when $\alpha$ is close to $1$, the asymptotic second-order convergence appears more slowly. The theoretical study of the convergence order is a subject of our future work.

\begin{figure}[H]
\centering

\begin{subfigure}{\textwidth}
	\begin{minipage}{.5\textwidth}
		\centering
		\includegraphics[width=\linewidth]{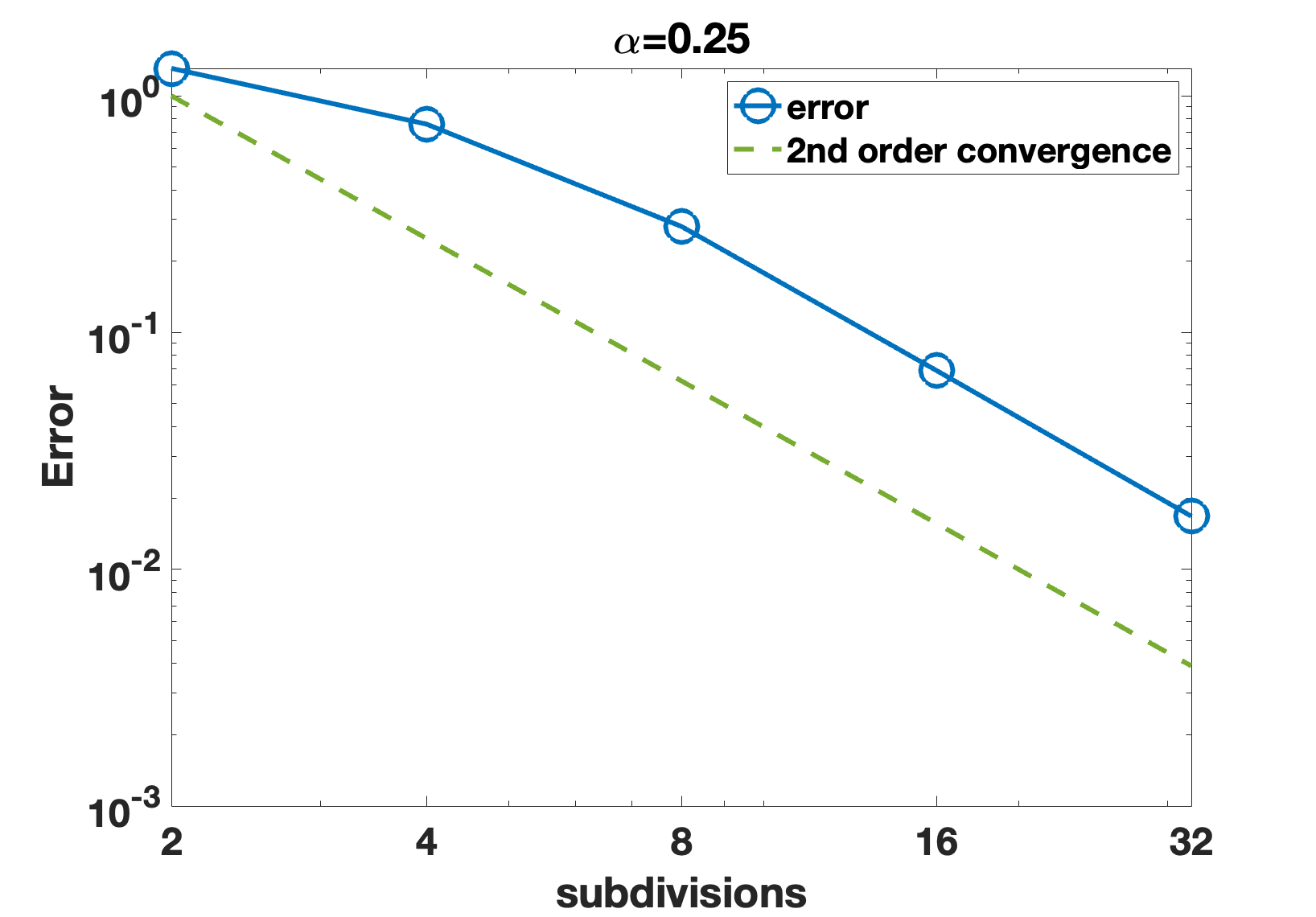}
	\end{minipage}%
	\begin{minipage}{.5\textwidth}
		\centering
		\includegraphics[width=\linewidth]{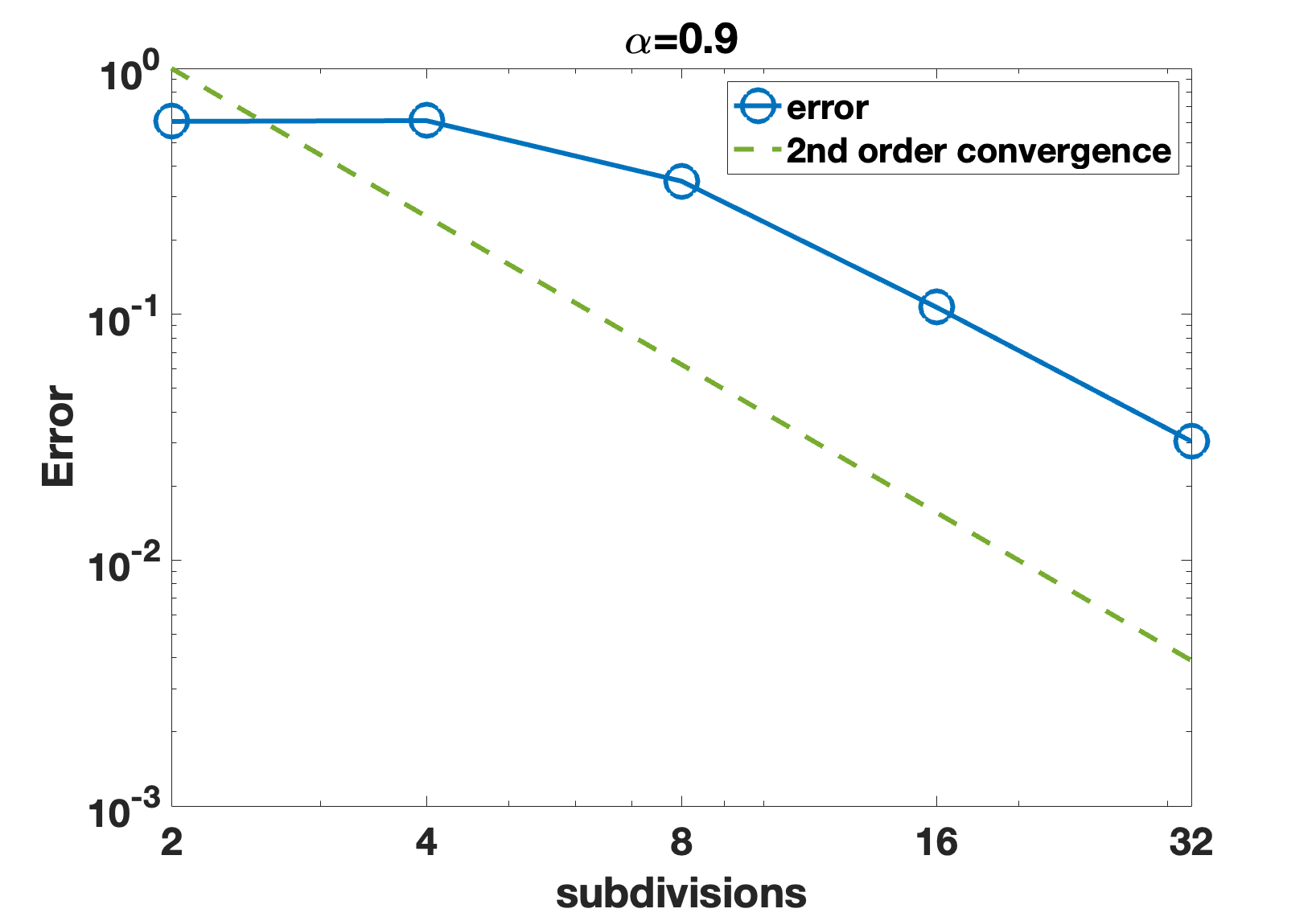}
	\end{minipage}
	\caption{Discrete fractional gradient error for function $f(x,y,z)$}
\end{subfigure}

\begin{subfigure}{\textwidth}
	\begin{minipage}{.5\textwidth}
		\centering
		\includegraphics[width=\linewidth]{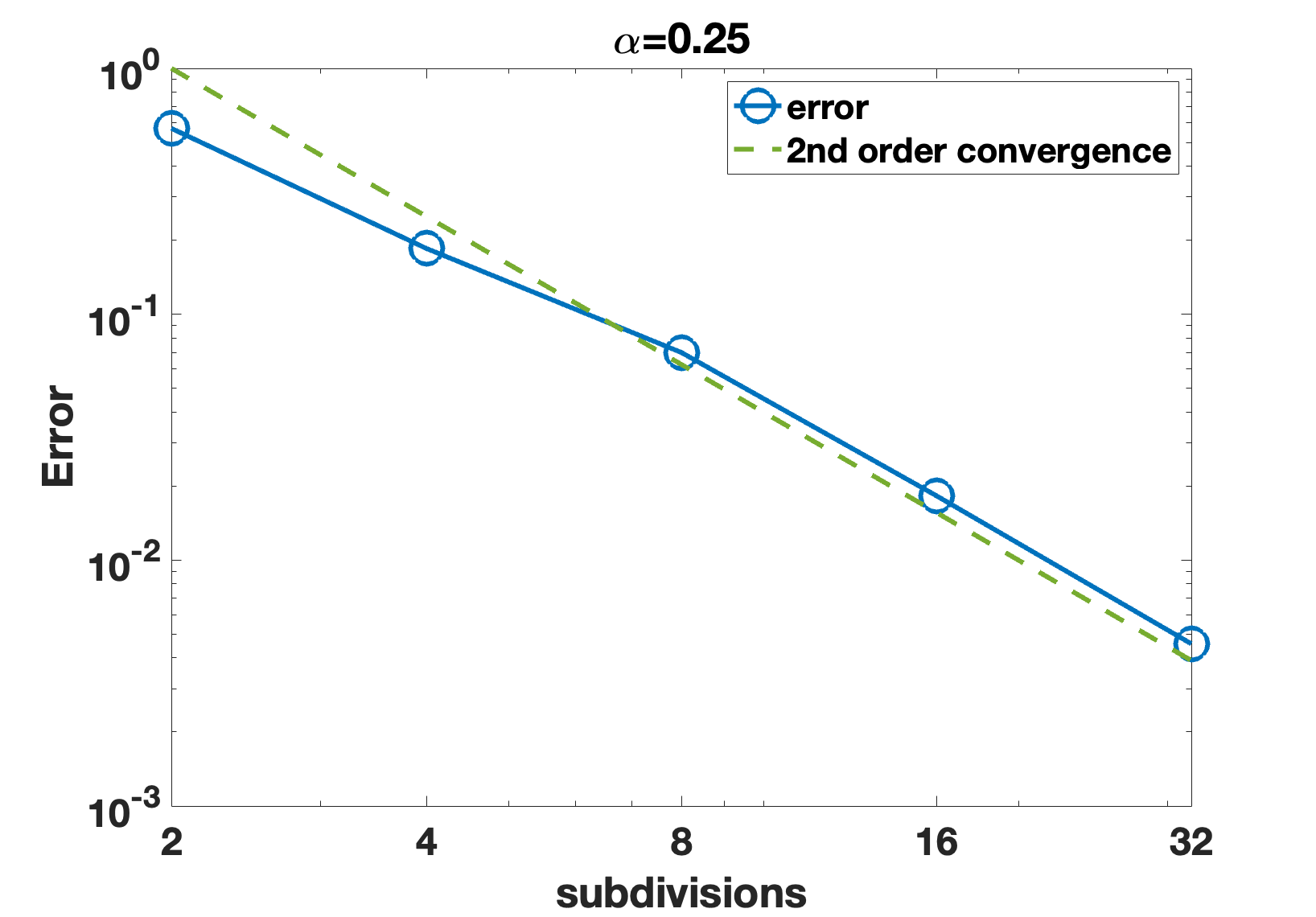}
	\end{minipage}%
	\begin{minipage}{.5\textwidth}
		\centering
		\includegraphics[width=\linewidth]{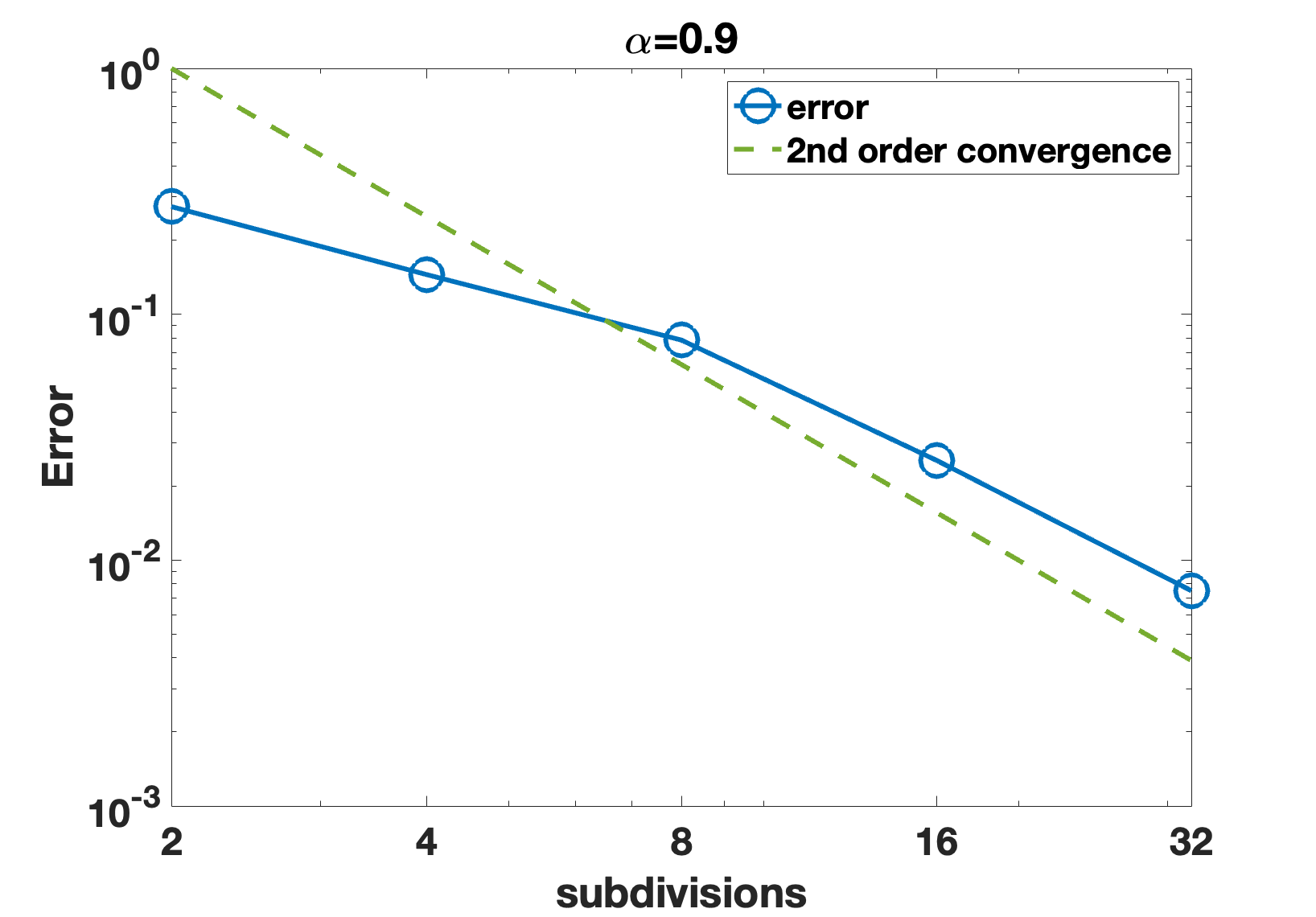}
	\end{minipage}
	\caption{Discrete fractional curl error for function $\bm{F}(x, y, z)$}
\end{subfigure}

\begin{subfigure}{\textwidth}
	\begin{minipage}{.5\textwidth}
		\centering
		\includegraphics[width=\linewidth]{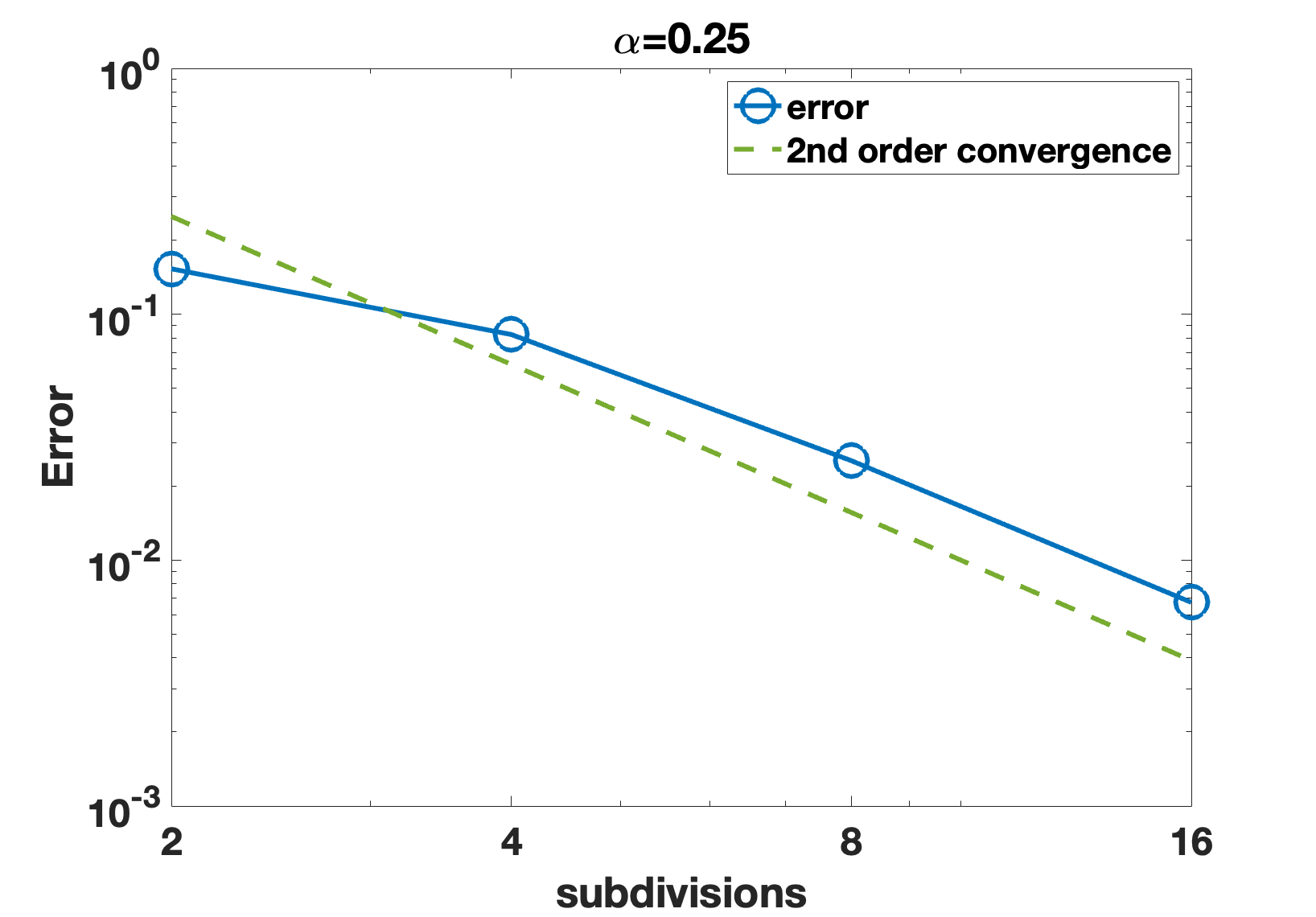}
	\end{minipage}%
	\begin{minipage}{.5\textwidth}
		\centering
		\includegraphics[width=\linewidth]{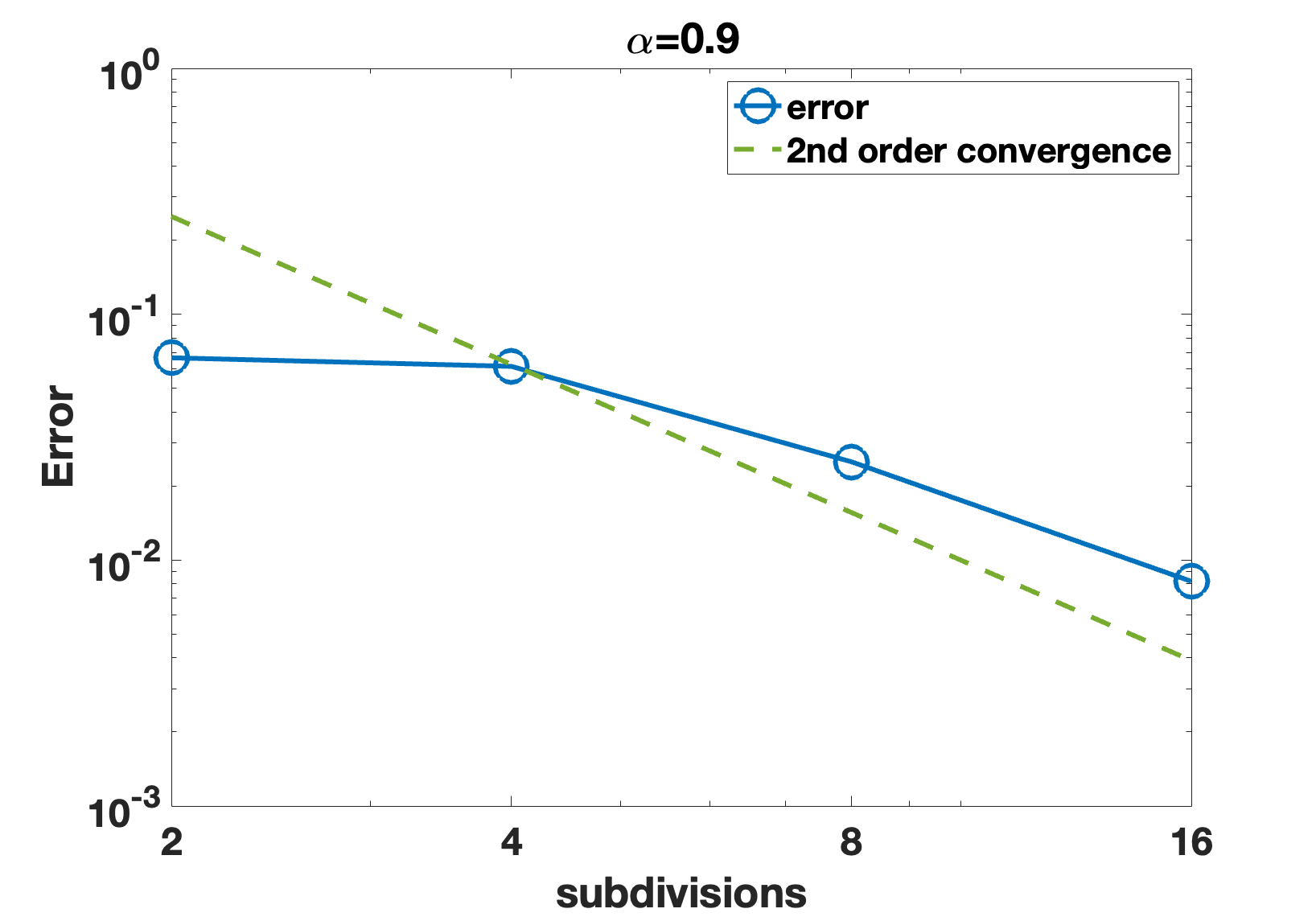}
	\end{minipage}
	\caption{Discrete fractional div error for function $\bm{F}(x, y, z)$}
\end{subfigure}

\caption{Error of discrete fractional gradient, curl, and divergence to their continuous counterparts against the number of subdivisions. Left column: $\alpha=0.25$. Right column: $\alpha=0.9$.} \label{fig:convergence_h}

\end{figure}

\begin{figure}[H]
\centering
\begin{subfigure}{\textwidth}
	\begin{minipage}{.5\textwidth}
		\centering
		\includegraphics[width=\linewidth]{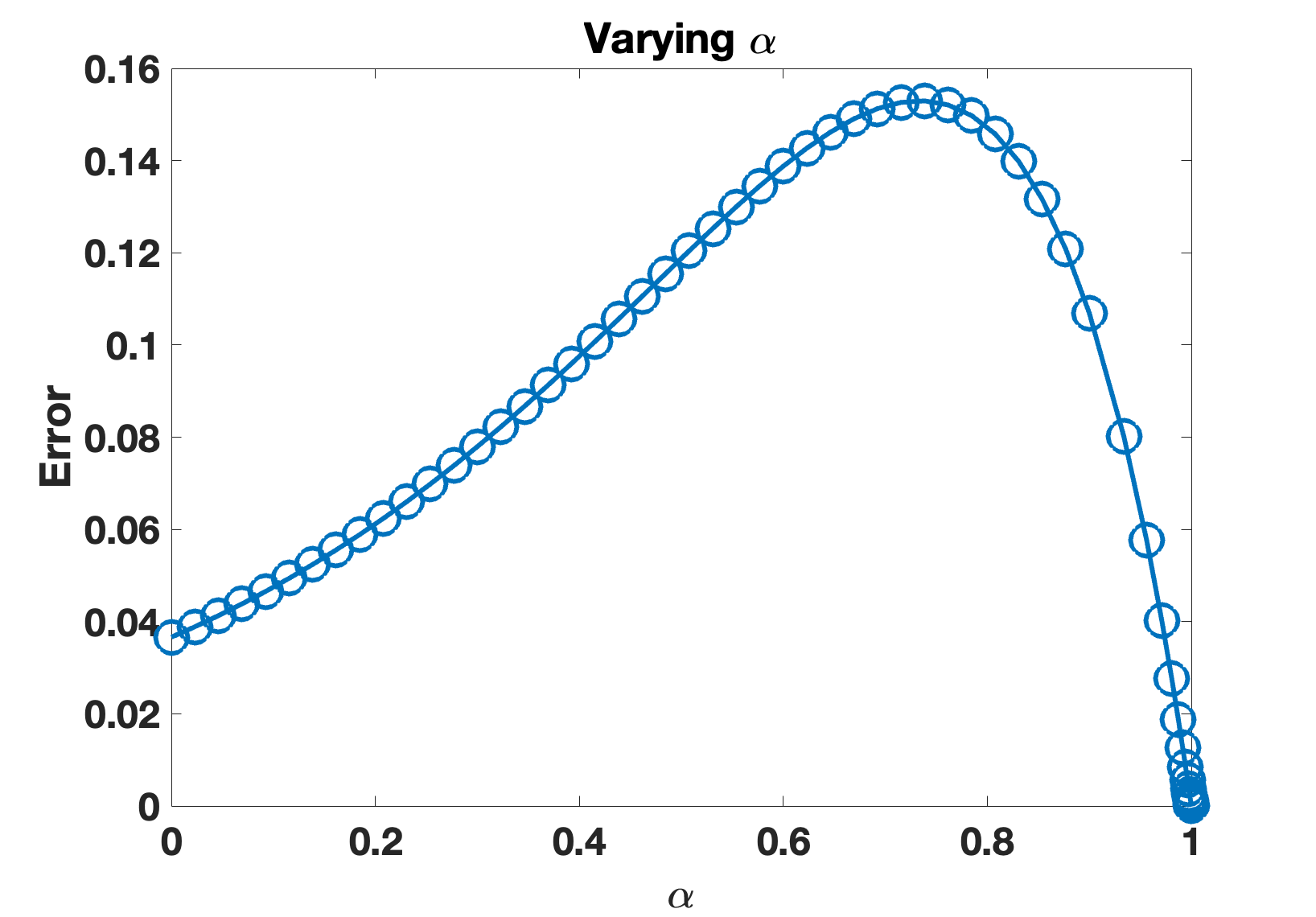}
	\end{minipage}%
	\begin{minipage}{.5\textwidth}
		\centering
		\includegraphics[width=\linewidth]{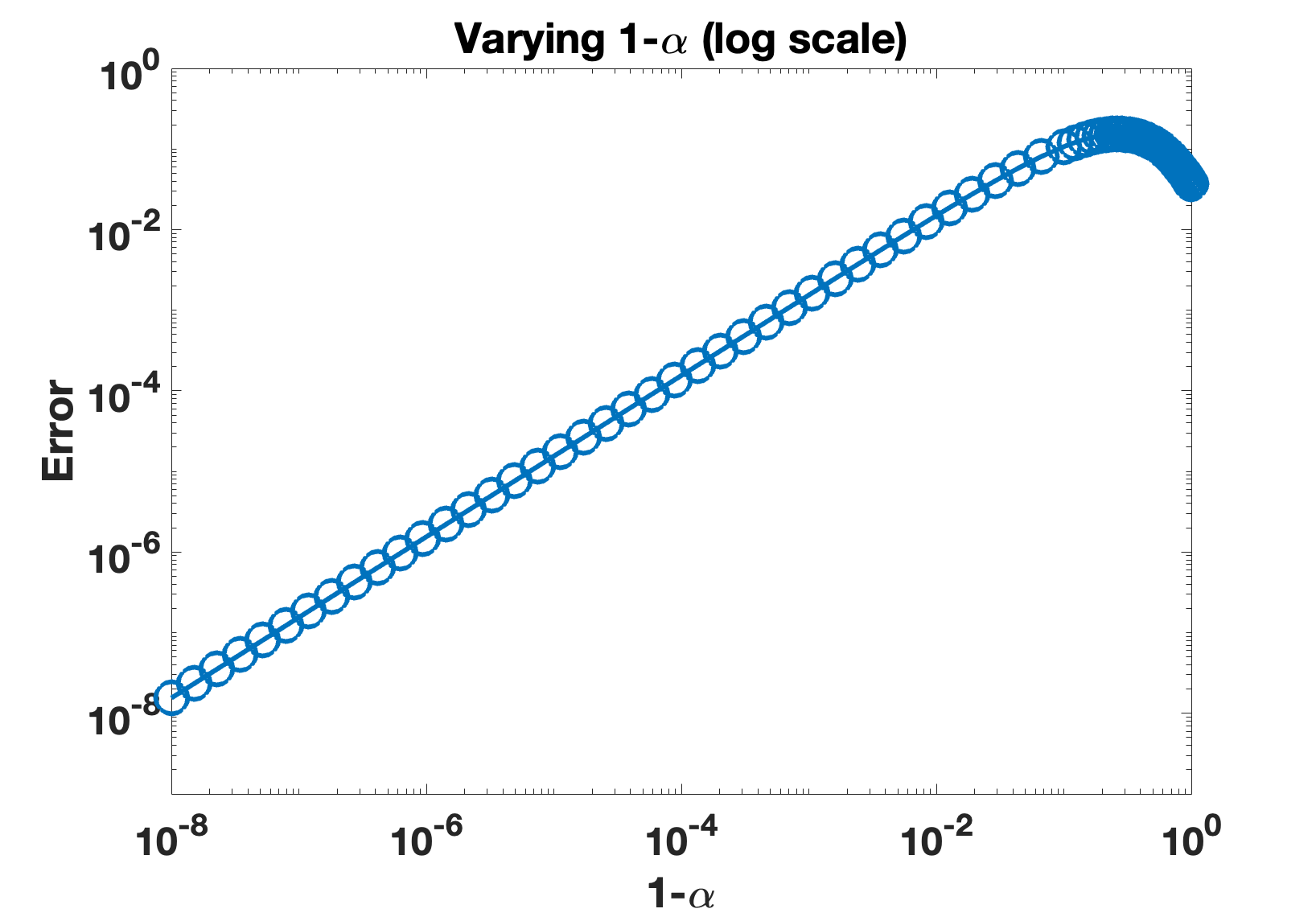}
	\end{minipage}
	\caption{Discrete fractional gradient error for function $f(x,y,z)$, at $16$ subdivisions.}
\end{subfigure}

\begin{subfigure}{\textwidth}
	\centering
	\begin{minipage}{.5\textwidth}
		\centering
		\includegraphics[width=\linewidth]{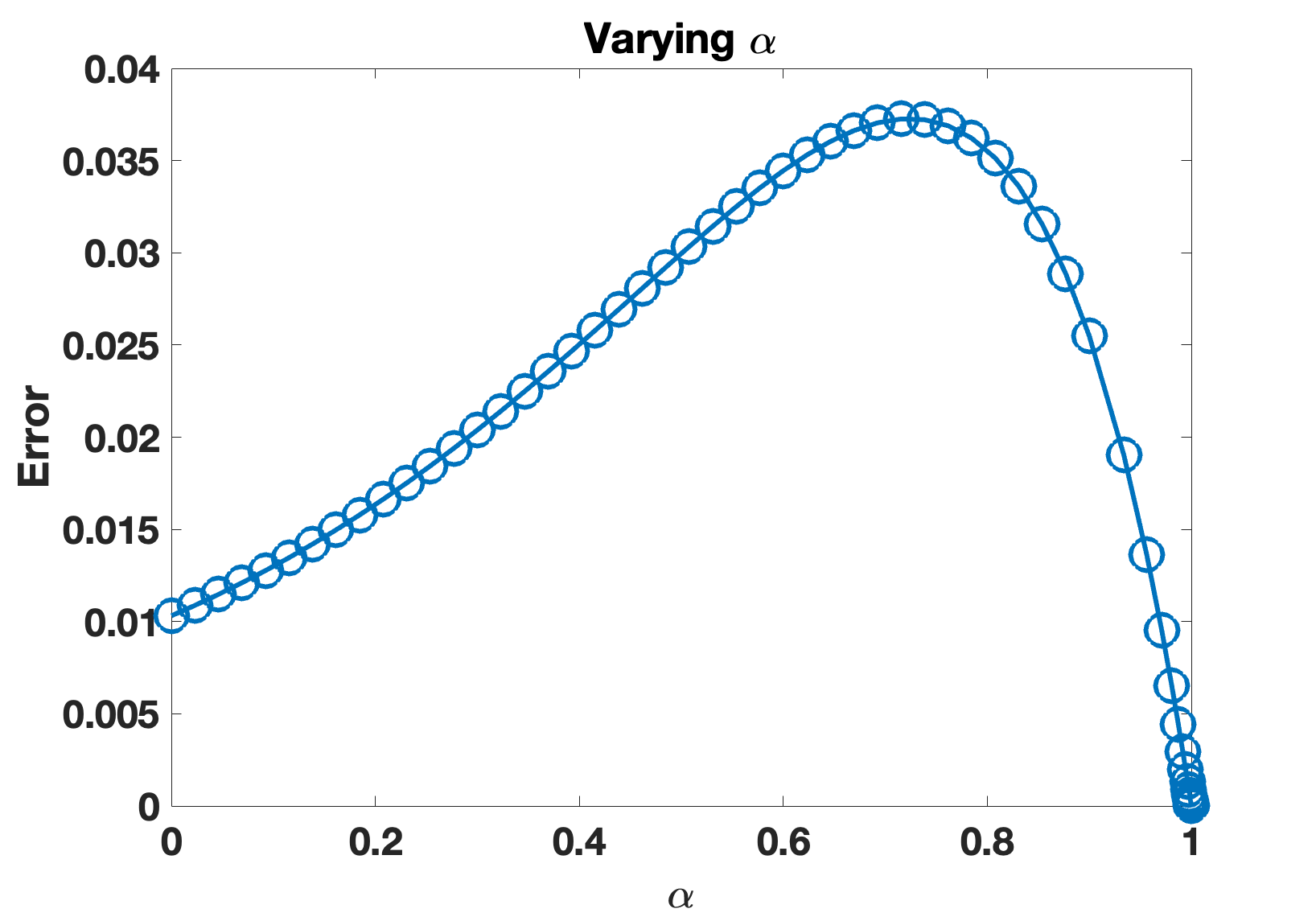}
	\end{minipage}%
	\begin{minipage}{.5\textwidth}
		\centering
		\includegraphics[width=\linewidth]{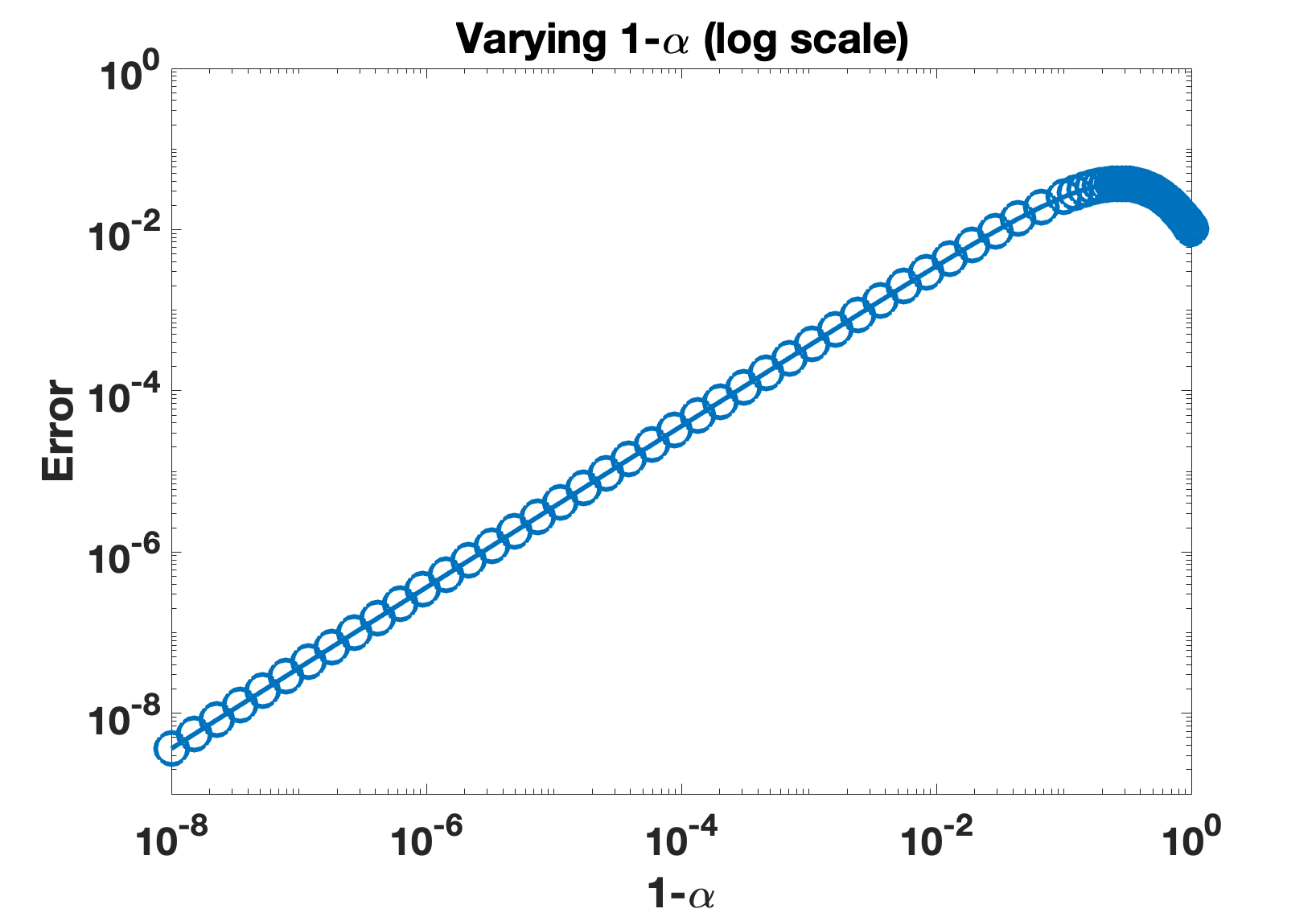}
	\end{minipage}
	\caption{Discrete fractional curl error for function $\bm{F}(x, y, z)$, at $16$ subdivisions.}
\end{subfigure}

\begin{subfigure}{\textwidth}
	\centering
	\begin{minipage}{.5\textwidth}
		\centering
		\includegraphics[width=\linewidth]{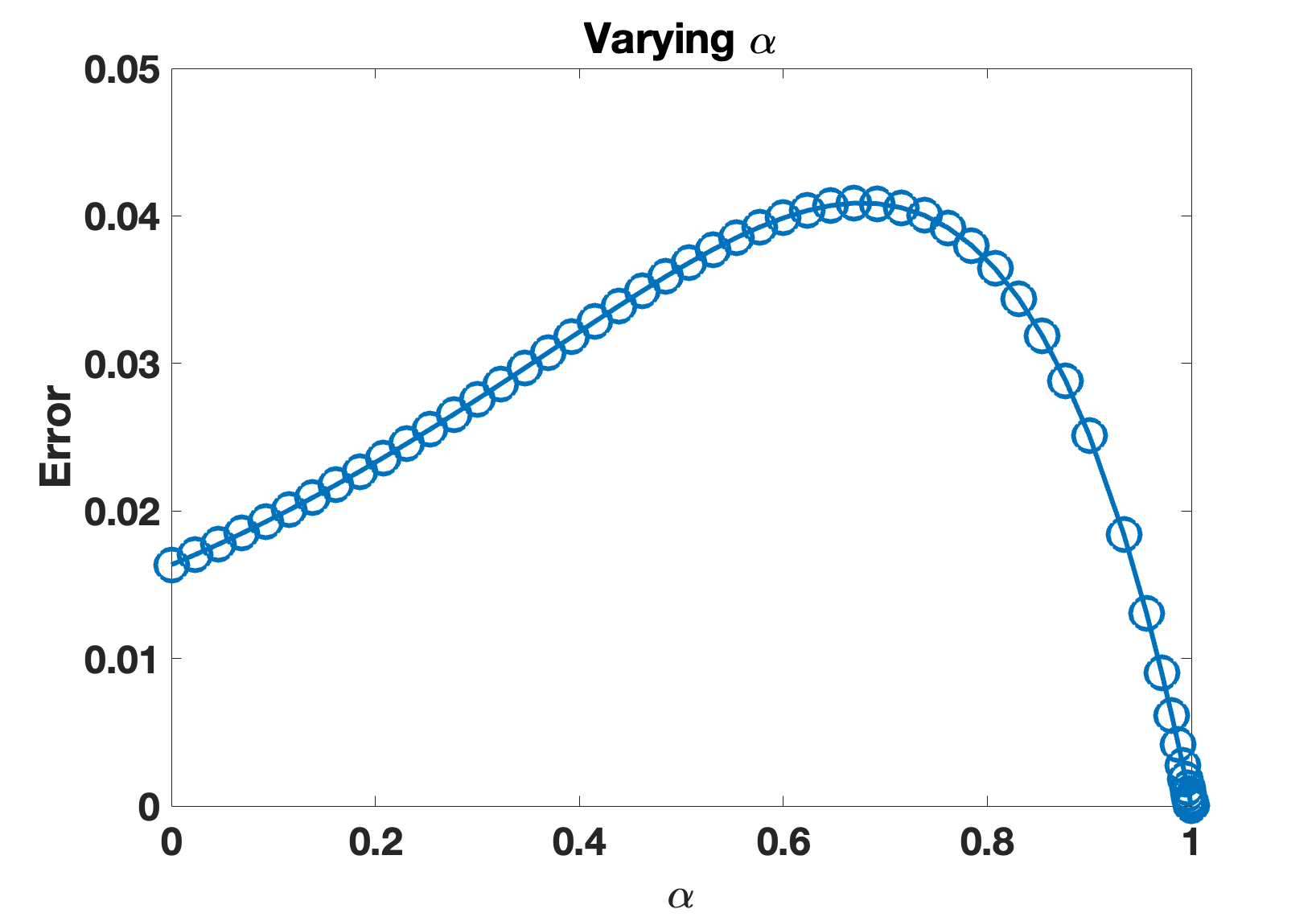}
	\end{minipage}%
	\begin{minipage}{.5\textwidth}
		\centering
		\includegraphics[width=\linewidth]{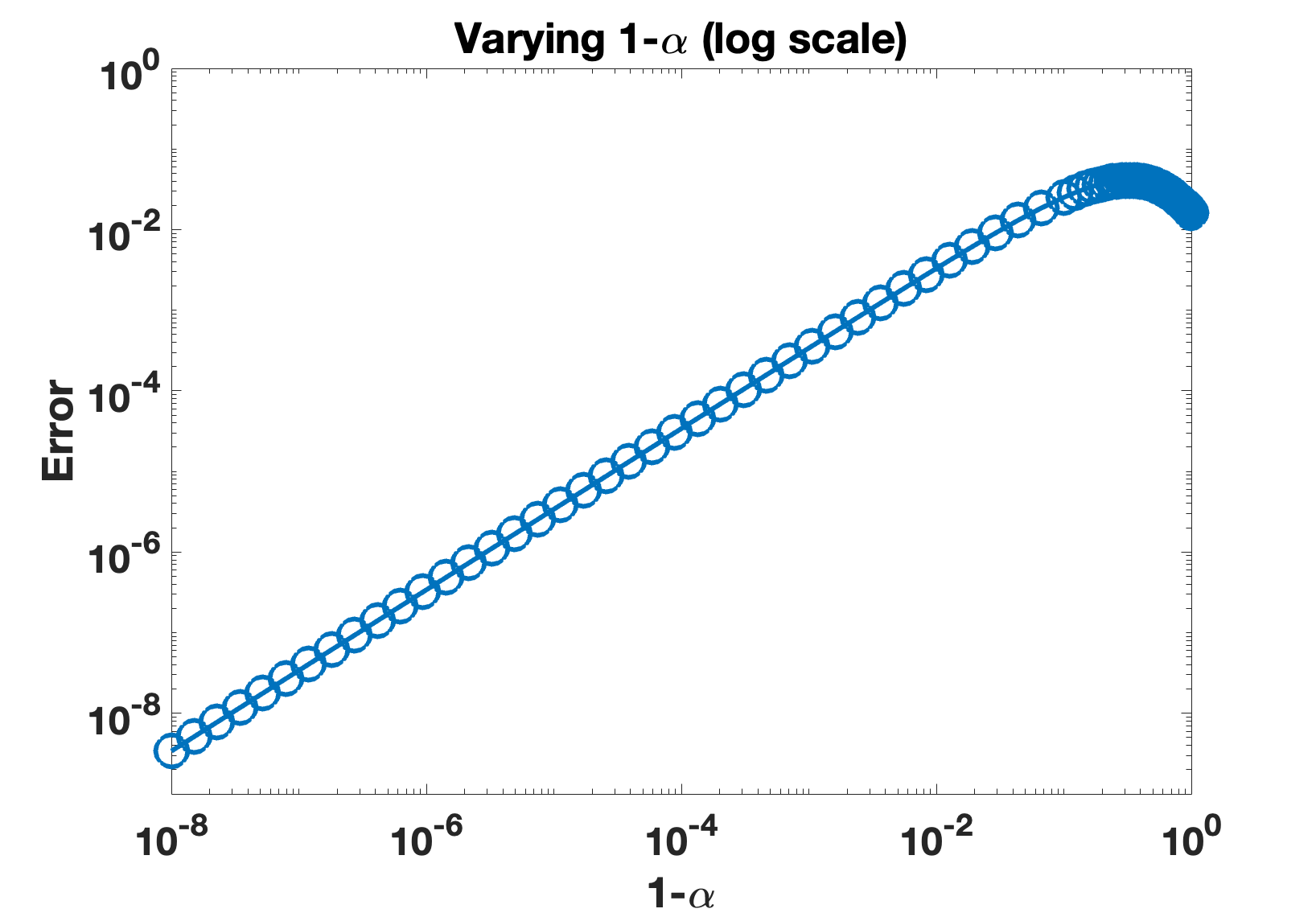}
	\end{minipage}
	\caption{Discrete fractional divergence error for function $\bm{F}(x, y, z)$, at $8$ subdivisions.}
\end{subfigure}

\caption{Error of discrete fractional gradient, curl, and divergence to their continuous counterparts against the fractional order $\alpha$. Left column: error versus $\alpha$ in linear scale. Right column: error versus $1-\alpha$ in log scale.} \label{fig:convergence_alpha}

\end{figure}

On the other hand, Figure~\ref{fig:convergence_alpha} plots the error (\ref{errorDpalpha}) for $p=0,1,2$ against $\alpha$ for a fixed number of subdivisions ($n=16$ for $p=0,1$ and $n=8$ for $p=2$). The plots in Figure~\ref{fig:convergence_alpha} deserve some explanation. On the left column plots, the error (in linear scale) is plotted against $\alpha$ (in linear scale), showing the relationship between the error and $\alpha$. One can see in these plots that the error approaches $0$ as $\alpha$ approaches $1$. This makes sense, since we have that $\mathbb{D}_p^\alpha \mathcal{R}_p = \mathcal{R}_{p+1} d_p^\alpha$ when $\alpha = 1$.  The plots on the right-hand sides of Figure~\ref{fig:convergence_alpha} more clearly show this phenomenon: the $x$-axis is $1-\alpha$, which is how far $\alpha$ is to $1$, and it is in log scale in order to clearly show values of $1-\alpha$ that are very close to $0$, i.e., values of $\alpha$ that are very close to $1$. The $y$-axis is the error, also in log scale, resulting in a clear linear relationship on the plot for $1-\alpha$ sufficiently small.

\subsection{Numerical verification of $\mathbb{D}_{p+1}^\alpha \mathbb{D}_p^\alpha = 0$} \label{sec:dd0}

\begin{figure}
	\centering
		\begin{minipage}{.5\textwidth}
			\centering
			\includegraphics[width=\linewidth]{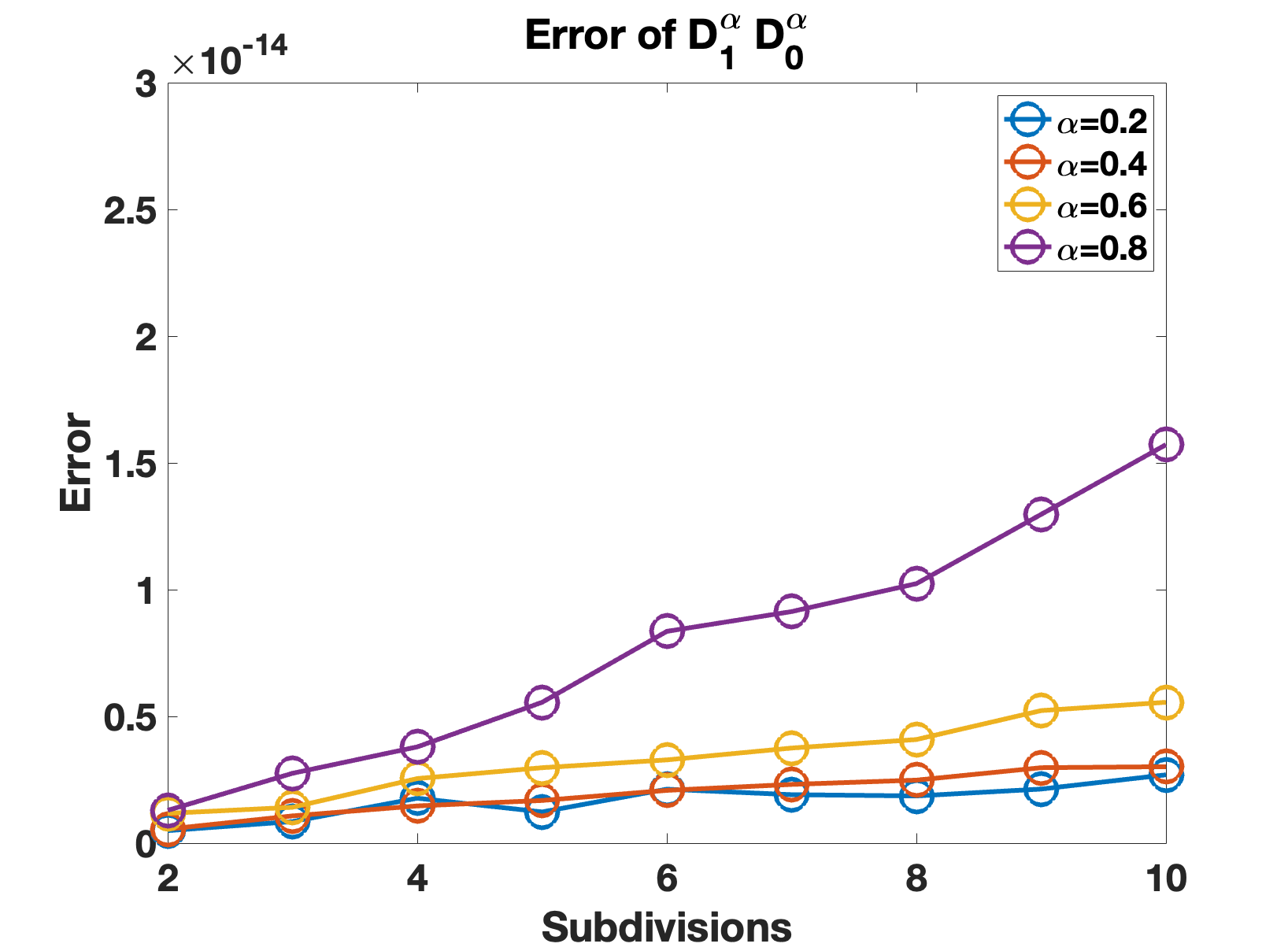}
		\end{minipage}%
		\begin{minipage}{.5\textwidth}
			\centering
			\includegraphics[width=\linewidth]{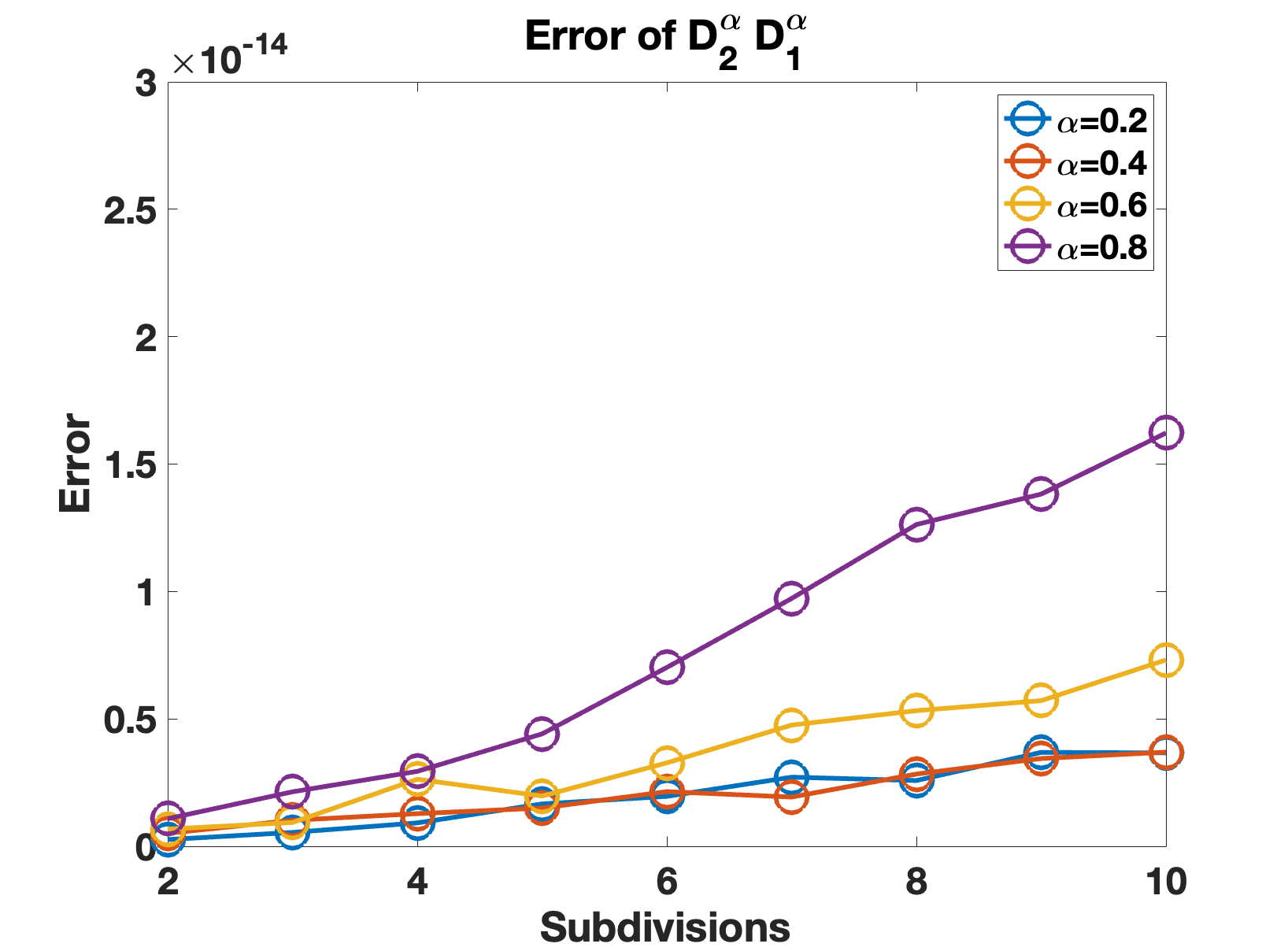}
		\end{minipage}	
	\caption{Numerical verification of $\mathbb{D}_{p+1}^\alpha \mathbb{D}_p^\alpha \textbf{c}^p = 0$} \label{fig:Dp1Dp}
	
\end{figure}

In this section, we numerically verify $\mathbb{D}_1^\alpha \mathbb{D}_0^\alpha = 0$ and $\mathbb{D}_2^\alpha \mathbb{D}_1^\alpha = 0$ by computing the errors in the same way as was done in Section~\ref{sec:convergence}. Concretely, the error is calculated as
\[
	\operatorname{RMS}(\mathbb{V}_{p+2}^{-1} \mathbb{D}_{p+1}^\alpha (\mathbb{D}_p^\alpha \mathcal{R}_p \omega^p))
\]
where we used the same scalar and vector field as in Section~\ref{sec:convergence} for $\omega^0$ and $\omega^1$.
Although the operators were mathematically defined such that $\mathbb{D}_{p+1}^\alpha \mathbb{D}_p^\alpha  = 0$ holds exactly, small numerical errors are expected in practice due to the floating point errors.  This is confirmed by Figure~\ref{fig:Dp1Dp}. As we can see, for different values of $\alpha$, the errors slightly increase when the number of subdivisions increases. However, the magnitude of the errors is of the order $10^{-14}$, which numerically verifies that $\mathbb{D}_{p+1}^\alpha \mathbb{D}_p^\alpha = 0$.

\section{Conclusion and future work}\label{sec:conclusion}

%%%%%%%%%%%%% SHORT VERSION (PAPER) %%%%%%%%%%%%%
In this paper, a type of fractional vector calculus was discretized on a 3D regular cubical complex using discrete exterior calculus. To do this, Tarasov's standard basis directional FVC operators were re-formulated so that they could be discretized using DEC. Discretizing the reformulated FVC operators led to the FDEC operators $\mathbb{D}_p^\alpha = \dfraciint{1-\alpha}{p+1} \, \mathbb{D}_p \, (\dfraciint{1-\alpha}{p})^{-1}$, which are structure-preserving in the sense that $\mathbb{D}_{p+1}^\alpha \, \mathbb{D}_p^\alpha = 0$, just as $d^\alpha_{p+1} \, d^\alpha_p = 0$. In addition, our FDEC operators involve relatively sparse matrices and accurately approximate the corresponding continuous operators numerically with a second-order convergence rate in the RMS error.

For future work, firstly, we would like to generalize our FDEC operators beyond cubical complexes to arbitrary cell complexes or simplicial complexes. Secondly, besides T-FVC considered in this work, it is natural to consider other types of FVC (possibly more generally nonlocal vector calculus) and their structure-preserving discretizations. Finally, although our approach involves relatively sparse matrices, the sparsity does decrease with increasing $p$. Therefore, we plan to investigate whether it is possible to compute the $\mathbb{D}_p^\alpha$ matrices in a more memory-efficient manner for practical implementations.

\section*{Acknowledgments}
This work is partially supported by the National Science Foundation under grant DMS-2208267. 

%\pagebreak
\bibliographystyle{elsarticle-num}
\bibliography{paper}

\begin{thebibliography}{10}
\expandafter\ifx\csname url\endcsname\relax
  \def\url#1{\texttt{#1}}\fi
\expandafter\ifx\csname urlprefix\endcsname\relax\def\urlprefix{URL }\fi
\expandafter\ifx\csname href\endcsname\relax
  \def\href#1#2{#2} \def\path#1{#1}\fi

\bibitem{metzler1994fractional}
R.~Metzler, W.~G. Gl{\"o}ckle, T.~F. Nonnenmacher, Fractional model equation
  for anomalous diffusion, Physica A: Statistical Mechanics and its
  Applications 211~(1) (1994) 13--24.

\bibitem{SokolovAnomalousDiffusion}
I.~M. Sokolov, J.~Klafter, From diffusion to anomalous diffusion: A century
  after einstein’s brownian motion, Chaos: An Interdisciplinary Journal of
  Nonlinear Science 15~(2) (2005) 026103.

\bibitem{dos2019analytic}
M.~A. dos Santos, Analytic approaches of the anomalous diffusion: A review,
  Chaos, Solitons \& Fractals 124 (2019) 86--96.

\bibitem{oliveira2019anomalous}
F.~A. Oliveira, R.~M. Ferreira, L.~C. Lapas, M.~H. Vainstein, Anomalous
  diffusion: A basic mechanism for the evolution of inhomogeneous systems,
  Frontiers in Physics 7 (2019) 18.

\bibitem{evangelista2018fractional}
L.~R. Evangelista, E.~K. Lenzi, Fractional diffusion equations and anomalous
  diffusion, Cambridge University Press, 2018.

\bibitem{TARASOV20082756}
V.~E. Tarasov, Fractional vector calculus and fractional maxwell’s equations,
  Annals of Physics 323~(11) (2008) 2756--2778.

\bibitem{baleanu2009fractional}
D.~Baleanu, A.~K. Golmankhaneh, A.~K. Golmankhaneh, M.~C. Baleanu, Fractional
  electromagnetic equations using fractional forms, International Journal of
  Theoretical Physics 48~(11) (2009) 3114--3123.

\bibitem{ortigueira2015fractionalMaxwell}
M.~D. Ortigueira, M.~Rivero, J.~J. Trujillo, From a generalised helmholtz
  decomposition theorem to fractional maxwell equations, Communications in
  Nonlinear Science and Numerical Simulation 22~(1-3) (2015) 1036--1049.

\bibitem{benson2000application}
D.~A. Benson, S.~W. Wheatcraft, M.~M. Meerschaert, Application of a fractional
  advection-dispersion equation, Water resources research 36~(6) (2000)
  1403--1412.

\bibitem{pang2019fpinns}
G.~Pang, L.~Lu, G.~E. Karniadakis, fpinns: Fractional physics-informed neural
  networks, SIAM Journal on Scientific Computing 41~(4) (2019) A2603--A2626.

\bibitem{d2014multidimensional}
M.~D'Ovidio, R.~Garra, Multidimensional fractional advection-dispersion
  equations and related stochastic processes, Electronic Journal of Probability
  19 (2014) 1--31.

\bibitem{meerschaert2006fractional}
M.~M. Meerschaert, J.~Mortensen, S.~W. Wheatcraft, Fractional vector calculus
  for fractional advection--dispersion, Physica A: Statistical Mechanics and
  its Applications 367 (2006) 181--190.

\bibitem{gatto2015numerical}
P.~Gatto, J.~S. Hesthaven, Numerical approximation of the fractional laplacian
  via $hp$-finite elements, with an application to image denoising, Journal of
  Scientific Computing 65~(1) (2015) 249--270.

\bibitem{scalas2000fractional}
E.~Scalas, R.~Gorenflo, F.~Mainardi, Fractional calculus and continuous-time
  finance, Physica A: Statistical Mechanics and its Applications 284~(1-4)
  (2000) 376--384.

\bibitem{wei2020generalization}
Y.~Wei, Y.~Kang, W.~Yin, Y.~Wang, Generalization of the gradient method with
  fractional order gradient direction, Journal of the Franklin Institute
  357~(4) (2020) 2514--2532.

\bibitem{wang2010direct}
H.~Wang, K.~Wang, T.~Sircar, A direct o(n log2 n) finite difference method for
  fractional diffusion equations, Journal of Computational Physics 229~(21)
  (2010) 8095--8104.

\bibitem{meerschaert2004vector}
M.~M. Meerschaert, J.~Mortensen, H.-P. Scheffler, Vector grunwald formula for
  fractional derivatives, FRACTIONAL CALCULUS AND APPLIED ANALYSIS. 7~(1)
  (2004) 61--82.

\bibitem{PANG2013597}
G.~Pang, W.~Chen, K.~Sze, Gauss–jacobi-type quadrature rules for fractional
  directional integrals, Computers \& Mathematics with Applications 66~(5)
  (2013) 597--607, fractional Differentiation and its Applications.

\bibitem{song2017computing}
F.~Song, C.~Xu, G.~E. Karniadakis, Computing fractional laplacians on
  complex-geometry domains: algorithms and simulations, SIAM Journal on
  Scientific Computing 39~(4) (2017) A1320--A1344.

\bibitem{leok2004foundations}
M.~Leok, Foundations of computational geometric mechanics, California Institute
  of Technology, 2004.

\bibitem{mullen2011discrete}
P.~Mullen, A.~McKenzie, D.~Pavlov, L.~Durant, Y.~Tong, E.~Kanso, J.~E. Marsden,
  M.~Desbrun, Discrete lie advection of differential forms, Foundations of
  Computational Mathematics 11~(2) (2011) 131--149.

\bibitem{elcott2007stable}
S.~Elcott, Y.~Tong, E.~Kanso, P.~Schr{\"o}der, M.~Desbrun, Stable,
  circulation-preserving, simplicial fluids, ACM Transactions on Graphics (TOG)
  26~(1) (2007) 4--es.

\bibitem{dominitz2009texture}
A.~Dominitz, A.~Tannenbaum, Texture mapping via optimal mass transport, IEEE
  transactions on visualization and computer graphics 16~(3) (2009) 419--433.

\bibitem{hormann2007mesh}
K.~Hormann, B.~L{\'e}vy, A.~Sheffer, Mesh parameterization: Theory and practice
  (2007).

\bibitem{CRUM201964}
J.~Crum, J.~A. Levine, A.~Gillette, Extending discrete exterior calculus to a
  fractional derivative, Computer-Aided Design 114 (2019) 64--72.

\bibitem{kilbas2006theory}
A.~Kilbas, A.~Trujillo, H.~Srivastava, J.~Trujillo, Theory and Applications of
  Fractional Differential Equations, no. v. 13 in North-Holland Mathematics
  Studies, Elsevier Science, 2006.

\bibitem{fracdiffeq}
K.~Diethelm, The Analysis of Fractional Differential Equations: An
  Application-Oriented Exposition Using Differential Operators of Caputo Type,
  Lecture Notes in Mathematics, Springer Berlin, Heidelberg, 2010.

\bibitem{UnifiedFVC}
M.~D’Elia, M.~Gulian, H.~Olson, G.~E. Karniadakis, Towards a unified theory
  of fractional and nonlocal vector calculus, Fractional Calculus and Applied
  Analysis 24~(5) (2021) 1301--1355.

\bibitem{ortigueira2018fractional}
M.~Ortigueira, J.~Machado, On fractional vectorial calculus, Bulletin of the
  Polish Academy of Sciences. Technical Sciences 66~(4) (2018).

\bibitem{fracPDEs1}
T.-T. Shieh, D.~E. Spector, On a new class of fractional partial differential
  equations, Advances in Calculus of Variations 8~(4) (2015) 321--336.

\bibitem{fracPDEs2}
T.-T. Shieh, D.~E. Spector, On a new class of fractional partial differential
  equations ii, Advances in Calculus of Variations 11~(3) (2018) 289--307.

\bibitem{vsilhavyfractional}
M.~{\v{S}}ilhav{\'y}, Fractional vector analysis based on invariance
  requirements (critique of coordinate approaches), Continuum Mechanics and
  Thermodynamics 32~(1) (2020) 207--228.

\bibitem{hirani2003discrete}
A.~N. Hirani, Discrete exterior calculus, California Institute of Technology,
  2003.

\bibitem{grady2010discrete}
L.~J. Grady, J.~R. Polimeni, Discrete calculus: Applied analysis on graphs for
  computational science, Vol.~3, Springer, 2010.

\bibitem{desbrun2005discrete}
M.~Desbrun, A.~N. Hirani, M.~Leok, J.~E. Marsden, Discrete exterior calculus,
  arXiv preprint math/0508341 (2005).

\bibitem{gillette2009notes}
A.~Gillette, Notes on discrete exterior calculus, University of Texas at Austin
  (2009).

\bibitem{teixeira2013differential}
F.~Teixeira, Differential forms in lattice field theories: An overview,
  International Scholarly Research Notices 2013 (2013).

\bibitem{cottrill2001fractional}
K.~Cottrill-Shepherd, M.~Naber, Fractional differential forms, Journal of
  Mathematical Physics 42~(5) (2001) 2203--2212.

\bibitem{vabishchevich2005finite}
P.~Vabishchevich, Finite-difference approximation of mathematical physics
  problems on irregular grids, Computational Methods in Applied Mathematics
  5~(3) (2005) 294--330.

\bibitem{da2014mimetic}
L.~B. da~Veiga, K.~Lipnikov, G.~Manzini, The mimetic finite difference method
  for elliptic problems, Vol.~11, Springer, 2014.

\bibitem{rodrigo2015finite}
C.~Rodrigo, F.~J. Gaspar, X.~Hu, L.~Zikatanov, A finite element framework for
  some mimetic finite difference discretizations, Computers \& Mathematics with
  Applications 70~(11) (2015) 2661--2673.

\bibitem{adler2021finite}
J.~H. Adler, C.~Cavanaugh, X.~Hu, L.~T. Zikatanov, A finite-element framework
  for a mimetic finite-difference discretization of maxwell's equations, SIAM
  Journal on Scientific Computing 43~(4) (2021) A2638--A2659.

\bibitem{wheatcraft2008fractional}
S.~W. Wheatcraft, M.~M. Meerschaert, Fractional conservation of mass, Advances
  in Water Resources 31~(10) (2008) 1377--1381.

\bibitem{tarasov2008fractional}
V.~E. Tarasov, Fractional equations of curie--von schweidler and gauss laws,
  Journal of Physics: Condensed Matter 20~(14) (2008) 145212.

\end{thebibliography}

\end{document}